\documentclass[11 pt]{amsart}

\usepackage{amssymb,enumerate,pinlabel,color}

\newtheorem{theorem}{Theorem}
\newtheorem{lem}{Lemma}
\newtheorem{prop}[theorem]{Proposition}
\theoremstyle{definition}
\newtheorem{defn}{Definition}
\newtheorem{exmp}{Example}

\newcommand{\supp}{\operatorname{supp}} 

\begin{document}

\subjclass{57M25, 57R58}

\keywords{(1,1) knot, 1-bridge torus knot, Heegaard diagram, Floer homology, train track}

\title[Doubly-pointed Heegaard diagrams compatible with (1,1) knots]{Constructing doubly-pointed Heegaard diagrams compatible with (1,1) knots}

\author{Philip Ording}
\address{\hskip-\parindent
Philip Ording\\
  Department of Mathematics\\
Medgar Evers College, CUNY.}
\email{pording@mec.cuny.edu}

\begin{abstract}
A $(1,1)$ knot $K$ in a $3$-manifold $M$ is a knot that intersects each solid torus of a genus 1 Heegaard splitting of $M$ in a single trivial arc.
Choi and Ko developed a parameterization of this family of knots by a four-tuple of integers, which they call Schubert's normal form.
This article presents an algorithm for constructing a doubly-pointed Heegaard diagram compatible with $K$, given a Schubert's normal form for $K$.
The construction, coupled with results of Ozsv{\'a}th and Szab{\'o}, provides a practical way to compute knot Floer homology groups for $(1,1)$ knots.
The construction uses train tracks, and its method is inspired by the work of Goda, Matsuda and Morifuji.
\end{abstract}

\thanks{Research partially supported by a PSC-CUNY grant}
\date{\today}
\maketitle

\section{Introduction}

Some objects of $3$-dimensional topology can be usefully characterized by objects of $2$-dimensional topology. 
A significant example is a Heegaard splitting, 
which is a closed $3$-manifold obtained by identifying the boundaries of two genus $g$ handlebodies by a homeomorphism.
The homeomorphism is completely determined by two sets $\alpha,\beta$ of $g$ curves in the splitting surface $\Sigma$, and we study the $3$-manifold in terms of the \emph{Heegaard diagram} $(\Sigma,\alpha,\beta)$.
The utility of Heegaard diagrams extends, by a slight modification, from manifolds to knots in manifolds. 
By marking two points $x,y\in \Sigma-\alpha-\beta$ in a given Heegaard diagram $(\Sigma,\alpha,\beta)$ for a $3$-manifold $M$, we specify a unique knot $K\subset M$.
The knot is comprised of the union of the two trivial arcs on either side of $\Sigma$ that join $x$ and $y$ while avoiding meridian disks bounded by $\alpha,\beta$. 
We say $(\Sigma,\alpha,\beta,x,y)$ is a \emph{doubly-pointed Heegaard diagram compatible with $K$}. 
This article adresses the problem of constructing a genus 1 doubly-pointed Heegaard diagram which is compatible with a given knot in the three-sphere.

\begin{figure}[ht!] 
\labellist 
\footnotesize\hair 2pt 
\pinlabel $\Sigma$ at 133 152
\pinlabel $\beta$ at 166 100
\pinlabel $\alpha$ at 294 7
\pinlabel $D_\beta$ at 336 137
\pinlabel $D_\alpha$ at 282 46
\pinlabel $t_\beta$ at 131 25
\pinlabel $t_\alpha$ at 228 175
\endlabellist 
\centering 
\includegraphics[width=\textwidth]{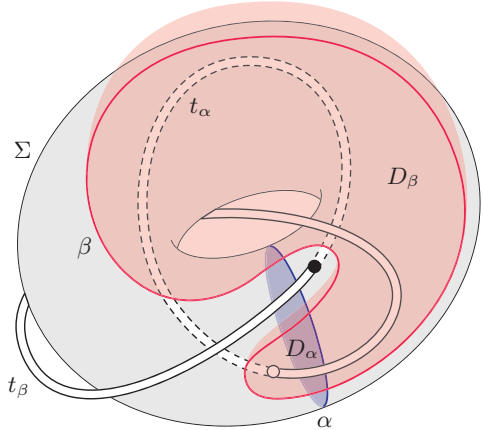}
\caption{A doubly-pointed Heegaard diagram compatible with the trefoil\label{fig:trefoil}}
\end{figure}

The motivation behind this question comes from the knot Floer homology developed by Ozsv{\'a}th and Szab{\'o}  and, independently, Rasmussen.
Knot Floer homology is a robust invariant; it produces the Alexander polynomial, the Seifert genus, fiberedness, a concordance invariant, uknot detection, and information about Dehn surgeries on a knot. 
For an introduction to the theory, see ~\cite{Ozsvath-Szabo:Intro}. 
The basic input for knot Floer homology is a Heegaard knot diagram, and Ozsv{\'a}th and Szab{\'o} showed that, in the case of doubly-pointed genus 1 Heegaard diagrams, the computation of knot Floer homology is combinatorial.
Although Manolescu, Ozsv\'ath and Sarkar found a combinatorial algorithm to calculate the Floer homology of any knot, the invariants are computationally more accessible given a genus 1 doubly-pointed Heegaard diagram for a knot as opposed to a multi-point and multi-curve Heegaard diagram that are used in the more general setting. 

Goda, Matsuda and Morifuji recognized that the set of knots admitting genus 1 doubly-pointed Heegaard diagrams correspond to $(g,b)$ knots with $g=b=1$ \cite{Goda-Matsuda-Morifuji}.
A $(1,1)$ knot $K$ in a $3$-manifold $M$ is a knot that intersects each solid torus of a genus 1 Heegaard splitting of $M$ in a single trivial arc.
The $(1,1)$ knots form a large family.
The torus knots and 2-bridge knots are proper subsets of the set of $(1,1)$ knots, and Fujii has shown that every Laurent polynomial satisfying the properties of an Alexander polynomial appears as the Alexander polynomial of some $(1,1)$ knot \cite{Fujii}.
Goda, Matsuda and Morifuji constructed numerous examples of genus 1 doubly-pointed Heegaard diagrams compatible with $(1,1)$ knots whose Floer homology was unknown, including several knots up to 10 crossings and certain pretzels \cite{Goda-Matsuda-Morifuji}. 
In each example, the attaching circle $\alpha$ remains fixed while $\beta$ is constructed via a sequence of steps beginning from a genus 1 doubly-pointed Heegaard diagram compatible with the trivial knot.  
Each step involves an isotopy of one of the trivial arcs and a simultaneous deformation of $\beta$, both of which are described by careful illustrations.
Using Schubert form, a parameterization of $(1,1)$ knots due to Choi and Ko \cite{Choi-Ko}, we are able to generalize Goda, Matsuda and Morifuji's examples to obtain an algorithm which constructs a genus 1 doubly-pointed Heegaard diagram for any $(1,1)$ knot with a given a Schubert form (Theorem \ref{thm:construction}).

\begin{figure} 
\labellist 
\footnotesize\hair 2pt 
\pinlabel $x$ at 257 107
\pinlabel $y$ at 245 51
\pinlabel $r$ at 139 94
\pinlabel $r$ at 382 62
\pinlabel $s$ at 140 77
\pinlabel $t$ at 131 24
\endlabellist 
\centering 
\includegraphics[width=\textwidth]{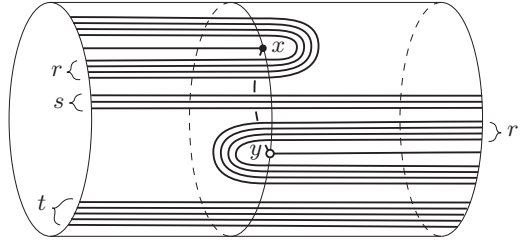}
\caption{Schubert form\label{fig:schubert}
}
\end{figure}

Rasmussen defined an integer 4-tuple   parameterization $K(p,q,r,s)$ of genus 1 doubly-pointed Heegaard diagrams ~\cite{Rasmussen:KnotHomologies}, and Doyle developed a computer program based on \cite{Goda-Matsuda-Morifuji} to compute the knot Floer homology groups up to relative Maslov grading for a given 4-tuple ~\cite{Doyle}. 
We use the complementarity between the $(1,1)$ knot trivial arcs and Heegaard diagram attaching circles to define an integer 4-tuple parameterization of genus 1 doubly-pointed Heegaard diagrams. 
Given a diagram resulting from our construction, we compute this parameterization, which we call Heegaard diagram Schubert form $HS(r',s',t',\rho')$ (Theorem \ref{thm:heegform}).
Using the correspondence below, which follows by direct inspection, we obtain the Rasmussen parameters required for Doyle's algorithm: 
$$
HS(r',s',t',\rho')=K(2r'+s'+t',r',s',-\rho'\mod{2r'+s'+t'}).
$$
Hedden observed that each step in the construction of Theorem \ref{thm:construction} either preserves or increases the total rank of the Floer homology ~\cite{Hedden:Private}. 
In particular, the absolute Maslov grading of the generators is unchanged throughout the construction of the Heegaard diagram. 
This continuity in the algorithm extends from the absolute grading of the unknot in the first step to the absolute Maslov grading for the knot whose Heegaard diagram appears in the final step.
This feature makes it possible to take advantage of Doyle's software and calculate the knot Floer homology groups exactly, including the absolute Maslov grading. 
The results of this paper and Doyle's software were employed in ~\cite{Ording}, to investigate the Floer invariants for satellite (1,1) knots.

The paper is organized as follows.
Section 2 summarizes the basic terminology of Heegaard diagrams, Schubert form, and train tracks, which offer a particularly convenient way to characterize curves in a surface.
Section 3 defines Schubert form in terms of train tracks.
In Section 4 we divide a $(1,1)$ knot into increments by which the Heegaard diagram construction of Section 5 proceeds. 
Section 6 computes the Heegaard diagram Schubert form. 
Examples 1--3 demonstrate all results in the simplest nontrivial case, the trefoil.
 
\subsection*{Acknowledgements} 
Matthew Hedden brought my attention to the problem of constructing Heegaard diagrams compatible with $(1,1)$ knots, and this research benefited from conversations with him, Walter Neumann, Peter Ozsv\'{a}th, Jacob Rasmussen, and Saul Schleimer.

\section{Preliminaries}\label{sec:prelims}

\subsection{Heegaard diagrams} 
A Heegaard splitting is a closed $3$-manifold $H_\alpha\cup_h H_\beta$ obtained by identifying the boundaries of two genus $g$ handlebodies $H_\alpha$ and $H_\beta$ by a homeomorphism $h:\partial H_\beta\to \partial H_\alpha$.
Let $\alpha\subset \partial H_\alpha$ and $\beta\subset\partial H_\beta$ be sets of $g$ meridional curves of $H_\alpha$ and $H_\beta$, respectively. 
The homeomorphism $h$ is determined (up to isotopy) by the attaching circles $h(\beta)\subset \partial H_\alpha$, (see, for example, Singer ~\cite{Singer}).  
A \emph{Heegaard diagram} $(\Sigma,\alpha,\beta)$ consists of the splitting surface $\Sigma=\partial H_\alpha=\partial H_\beta$  and the curves $\alpha$ and $h(\beta)$ in $\Sigma$.

Given a Heegaard diagram $(\Sigma,\alpha,\beta)$ for a $3$-manifold $M$, we specify a knot $K\subset M$ uniquely by marking two points $x,y\in \Sigma-\alpha-\beta$.
To identify $K$, let $D_\alpha\subset H_\alpha$ be a set of pairwise disjoint properly embedded disks, the boundaries of which correspond to the set of meridians $\alpha$.
A properly embedded arc $\gamma$ in a $3$-manifold $(M,\partial M)$ is trivial if there is an embedded disk $D$ such that $\partial D\cap\gamma=\gamma$ and $\partial D\cap\partial M$ is the arc $\partial D\setminus \text{int}(\gamma)$. 
There is a unique (up to isotopy) trivial arc $t_\alpha\subset H_\alpha-D_\alpha$ spanning $x,y$, since the complement of $D_\alpha$ in $H_\alpha$ is homeomorphic to a ball.  
Define $t_\beta$ similarly. 
If the closed curve $t_\alpha\cup t_\beta$ has knot type $K$, the $(\Sigma,\alpha,\beta,x,y)$ is a called a \emph{doubly-pointed Heegaard diagram compatible with $K$}. 

Figure \ref{fig:trefoil} depicts a genus 1 doubly-pointed Heegaard diagram compatible with the trefoil (in the three-sphere).

\subsection{(1,1) knots}
Doll defined a knot $K$ in a $3$-manifold $M$ to be \emph{genus-$g$ bridge-$b$}, or, simply $(g,b)$ if there is a genus $g$ Heegaard splitting of $M$ such that $K$ intersects each handlebody in $b$ trivial arcs \cite{Doll}.
Call $(H_\alpha,t_\alpha)\cup_h(H_\beta,t_\beta)$ a $(1,1)$ decomposition of a $(1,1)$ knot $K\subset M$ if $H_\alpha\cup_h H_\beta$ is a genus $1$ Heegaard splitting of $M$, $t_\alpha\subset H_\alpha$ and $t_\beta\subset H_\beta$ are properly embedded trivial arcs, and $K=t_\alpha\cup t_\beta$.
 
Choi and Ko \cite{Choi-Ko} demonstrated that
every $(1,1)$ knot is represented by an integer 4-tuple $(r,s,t,\rho)$, where $r$, $s$, and $t$ are non-negative.
In their presentation, called \emph{Schubert's normal form} (or simply \emph{Schubert form}), they isotope the arc $t_\beta$ of a $(1,1)$ decomposition into the  torus $\partial H_\alpha$. 
Then they prove that such an arc in the torus is isotopic to an arc in the form depicted in Figure \ref{fig:schubert}, where opposite ends of the cylinder are identified by a $\frac{2\pi\rho}{2r+s+t+1}$-rotation.
The parameters $r,s,t$ count the number of strands in each parallel class.
The arc $t_\alpha$ belongs to a ``standard'' meridian disk of the solid torus $H_\alpha$, as shown in Figure \ref{fig:schubert}.
The symbol $S(r,s,t,\rho)$ denotes the knot type of $t_\alpha\cup t_\beta$. 
Note that Schubert form is not unique;
for example, $S(0,0,2,2)$ and $S(1,1,0,-1)$ are identical trefoils.
 
\subsection{Train tracks}

Thurston originally introduced train tracks to study geodesic laminations of hyperbolic surfaces ~\cite{Thurston}, but we use the terminology merely in terms of the combinatorics of doubly-pointed Heegaard diagrams in the torus.
The following definitions are drawn from Penner and Harer \cite{Penner-Harer}, and we also rely, in part, on the exposition by Masur, Mosher, and Schleimer ~\cite{Masur-Mosher-Schleimer}. 

A \emph{train track with stops} (or simply a \emph{track}) is a graph $\tau$ that is properly embedded in a surface $\Sigma$ and satisfies the following conditions.
Edges of a train track, which are called \emph{branches}, are assumed to be smoothly embedded, and at each vertex the incident edges are mutually tangent. 
A vertex from which branches emanate in two directions (``incoming'' and ``outgoing'') is called a \emph{switch}, and a \emph{stop} is a vertex from which branches emanate in one direction.
Each vertex of a train track is either a switch in the interior of $\Sigma$ or a stop in $\partial \Sigma$.
There is at most one stop in each boundary component of the surface, and the one-sided tangent vector of each branch incident on a stop is transverse to $\partial \Sigma$.
We assume that at each switch exactly three branches meet, a pair of \emph{small} branch ends to one side and one \emph{large} branch end to the other.

Let $\mathcal{B}=\mathcal{B}(\tau)$ be the set of all branches of a track $\tau$. A \emph{tie neighborhood} $N(\tau)\subset \Sigma$ of $\tau$
is a union of rectangles $\{R_b:b\in\mathcal{B}\}$, each of which is foliated by vertical intervals, called \emph{ties}, that are transverse to $\tau$.
At each switch, the upper and lower thirds of the vertical side of the large branch end rectangle is identified with the two vertical sides of the small branch end rectangles.
It is required that each component of $\Sigma\setminus N(\tau)$ has at least one corner and the following index has a negative value on the component:
$$
\chi-\frac{1}{4}\#\,\text{outward-pointing corners}+\frac{1}{4}\#\,\text{inward-pointing corners}, 
$$
where $\chi$ is the Euler characteristic.
In particular, no branch is homotopic to another branch or a boundary component.

If $\gamma$ is a curve contained in $N(\tau)$ and transverse to the ties, then $\gamma$ is said to be \emph{carried} by $\tau$, and we write $\gamma\prec\tau$.
A \emph{transverse measure} $w$ on $\tau$ is a function which assigns a non-negative number, called a \emph{weight}, to each branch in such a way that the sums of the weights on the incoming and outgoing branches at each switch are equal.
Given an ordering on the branch set $\mathcal{B}=\{b_1,\dots,b_n\}$, the following notation will be convenient: $(w_1,\dots,w_n):=(w(b_1),\dots,w(b_n))$.
A collection of curves $\gamma\prec\tau$ defines a transverse measure on $\tau$, $w_\gamma(b):=|\gamma\cap t|$, where $t$ is any tie of $R_b$.
We call $w_\gamma$ the \emph{counting measure of $\gamma$ on $\tau$}.
Conversely, to every non-negative integer-valued function $w$ on $\mathcal{B}$ satisfying the switch condition, there is a collection of curves $\gamma\prec\tau$ with $w=w_\gamma$.
An \emph{extension} of a train track $\tau$ is a train track $\sigma$ which contains $\tau$ as a subset, and we write $\tau < \sigma$ (we also say $\tau$ is a \emph{subtrack} of $\sigma$).
The \emph{support} of an arc $\gamma\prec\sigma$ is the smallest subtrack $\supp\gamma<\sigma$ that carries $\gamma$.

\begin{figure}[ht!] 
\labellist 
\footnotesize\hair 2pt 
\pinlabel $b_1$ at 65 150
\pinlabel $b_2$ at 70 116
\pinlabel $b_3$ at 132 136
\pinlabel $b_1$ at 285 136
\pinlabel $b_2$ at 343 150
\pinlabel $b_3$ at 344 116
\pinlabel $b_1$ at 38 37
\pinlabel $b_2$ at 6 24
\pinlabel $b_3$ at 56 9
\pinlabel $b_1$ at 184 42
\pinlabel $b_2$ at 152 15
\pinlabel $b_3$ at 220 15
\pinlabel $b_1$ at 329 37
\pinlabel $b_2$ at 310 9
\pinlabel $b_3$ at 362 25
\endlabellist 
\centering 
\includegraphics[width=4in]{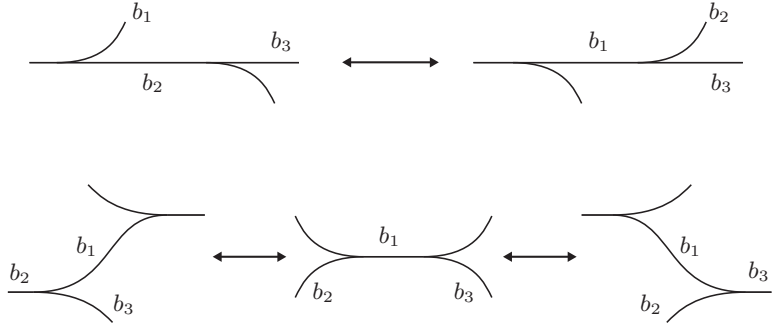}
\caption{Train track moves.\label{fig:moves}}
\end{figure}

Two measured train tracks $(\tau,w)$, $(\tau',w')$ are \emph{equivalent} if they are related by a sequence of isotopies and the following local moves. 
A \emph{left splitting} or \emph{right splitting} replaces a branch with large ends (a \emph{large branch}) by a branch with small ends (a \emph{small branch}).
Suppose that $w$ is a transverse measure on the train track $\tau$ before splitting at the branch $b_1$, shown in Figure \ref{fig:moves}, bottom center.
If $w_2\geq w_3$ ($w_2\leq w_3$), then $w$ induces a transverse measure $w'$ on the train track $\tau'$ resulting from the left (right) splitting of $\tau$ at $b_1$, and $w'_1$ equals $w_2-w_3$ ($w_3-w_2$), shown in Figure \ref{fig:moves}, bottom left (right). 
A \emph{slide} alters a branch having one large and one small end (a \emph{mixed} branch) according to Figure \ref{fig:moves}, top.
For a slide move, $w_1'=w_2+w_3$.
All other branch weights are unchanged.
The inverse of a slide is a slide, and the inverse of a splitting is called a \emph{fold}.

\section{Schubert train track}
 
Let $\Sigma$ be the closed oriented surface of genus 1.
Fix once and for all an oriented meridian-longitude pair $\mu, \lambda\subset\Sigma$ and a point $x\in\mu\setminus\lambda$.
Let $y:=\mu\cap\lambda$.
We will represent $\Sigma$ by the fundamental polygon $P=\lambda\mu\lambda^{-1}\mu^{-1}$, which is a rectangle with horizontals homologous to $\lambda$, oriented from right to left, and verticals homologous to $\mu$, oriented from bottom to top. 
Suppose $\hat{\Sigma}=\Sigma-D_x-D_y$, where $D_x$, $D_y\subset\Sigma$ are small disks whose boundary contains $x,y$, respectively.

\begin{defn} 
The \emph{Schubert train track} $\sigma$ denotes one of three train tracks with stops $\sigma_{--}$, $\sigma_-$, $\sigma_+$ in $\hat{\Sigma}$ that appear in Figure \ref{fig:traintrack}. 
The boundary component $D_x$ ($D_y$) and stop $x$ ($y$) are depicted by the small black (white) circle. 
\end{defn}

\begin{figure}[ht!] 
\labellist 
\footnotesize\hair 2pt 
\pinlabel $b_1$ at 21 98 
\pinlabel $b_2$ at 22 68
\pinlabel $b_3$ at 90 105
\pinlabel $b_4$ at 68 73
\pinlabel $b_5$ at 125 95
\pinlabel $\sigma_{--}$ at 87 5  
\pinlabel $b_1$ at 197 98
\pinlabel $b_2$ at 196 68
\pinlabel $b_3$ at 267 110
\pinlabel $b_4$ at 236 76 
\pinlabel $b_5$ at 300 95
\pinlabel $\sigma_{-}$ at 260 5
\pinlabel $b_1$ at 372 98
\pinlabel $b_2$ at 370 68
\pinlabel $b_3$ at 410 119
\pinlabel $b_4$ at 454 79 
\pinlabel $b_5$ at 476 95
\pinlabel $\sigma_{+}$ at 435 5 
\endlabellist 
\centering 
\includegraphics[width=\textwidth]{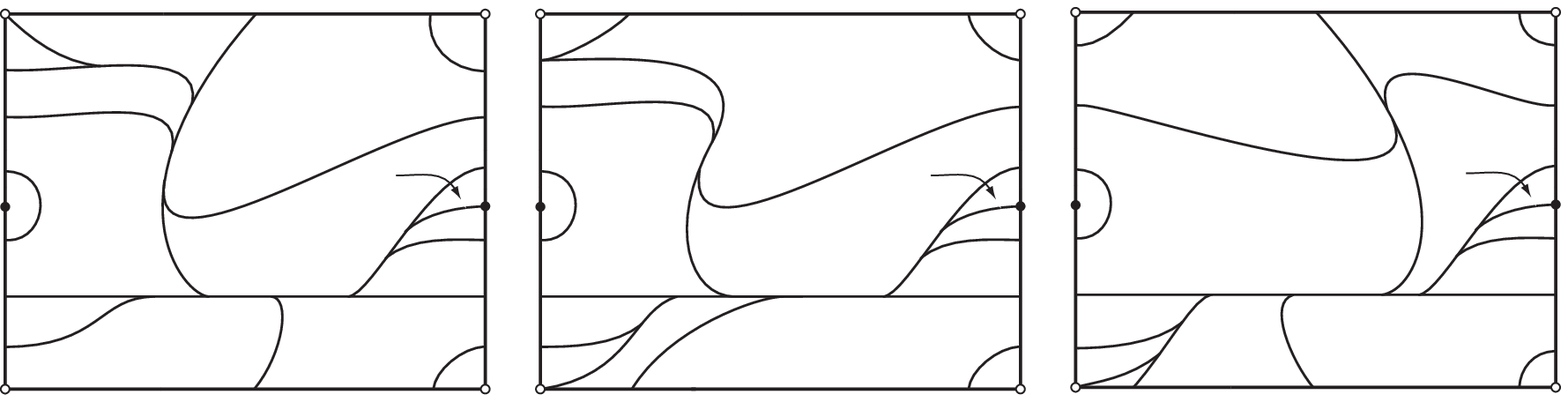} 
\caption{Schubert train track.\label{fig:traintrack}}
\end{figure} 

\begin{lem}\label{lem:traintrack}
The trivial arc $t_\beta$ of a $(1,1)$ knot with Schubert form $S(r,s,t,\rho)$ is carried by $\sigma$, where $\sigma$ is the Schubert train track $\sigma_{--}$ if $\rho<-r$, $\sigma_-$ if $-r\leq\rho<0$, and $\sigma_+$ if $\rho\geq 0$. 
The counting measure $w=w_{t_\beta}$ of $t_\beta$ on $\sigma$ is 
$$(w_1,w_2,w_3,w_4,w_5)=(r,s,t,|\rho|,1).$$
\end{lem}

\begin{proof}
From the definition of Schubert form $S(r,s,t,\rho)$, it follows that $t_\beta\prec\theta$, where $\theta$ is the train track with stops $\theta_-$ if $\rho<0$ and $\theta_+$ if $\rho\geq 0$, as shown in Figure \ref{fig:theta}.

\begin{figure}[ht!] 
\labellist 
\footnotesize\hair 2pt 
\pinlabel $b_1$ at 194 90
\pinlabel $b_2$ at 207 73
\pinlabel $b_3$ at 149 37
\pinlabel $b_4$ at 105 79
\pinlabel $b_5$ at 145 90
\pinlabel $\theta_-$ at 165 10
\pinlabel $b_1$ at 399 90
\pinlabel $b_2$ at 412 73
\pinlabel $b_3$ at 351 37
\pinlabel $b_4$ at 312 79
\pinlabel $b_5$ at 348 90
\pinlabel $\theta_+$ at 367 10
\endlabellist 
\centering 
\includegraphics[width=\textwidth]{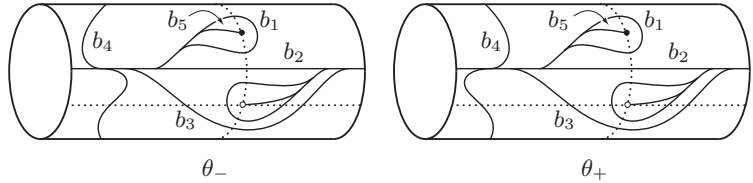}
\caption{Schubert form as measured train track with stops.\label{fig:theta}}
\end{figure}

Moreover, the counting measure $v=v_{t_\beta}$ of $t_\beta$ on $\theta$ is $(v_1,\dots,v_5)=(r,s,t,|\rho|,1)$. 
The equivalence of the measured train tracks $(\theta,v)$ and $(\sigma,w)$, results from the following sequence of train track moves taking $\theta$ to $\sigma$. 

\begin{figure}[ht!] 
\labellist 
\footnotesize\hair 2pt 
\pinlabel $b_1$ at 21 98 
\pinlabel $b_2$ at 22 65
\pinlabel $b_3$ at 99 121
\pinlabel $b_4$ at 58 77
\pinlabel $b_5$ at 125 95
\pinlabel $\theta'_{-}$ at 87 5  
\pinlabel $b_1$ at 195 98
\pinlabel $b_2$ at 196 65
\pinlabel $b_3$ at 267 119
\pinlabel $b_4$ at 232 79 
\pinlabel $b_5$ at 300 95
\pinlabel $\theta''_{-}$ at 260 5
\pinlabel $b_1$ at 369 98
\pinlabel $b_2$ at 368 65
\pinlabel $b_3$ at 444 121
\pinlabel $b_4$ at 402 79 
\pinlabel $b_5$ at 476 95
\pinlabel $\theta'''_{-}$ at 435 5 
\endlabellist 
\centering 
\includegraphics[width=\textwidth]{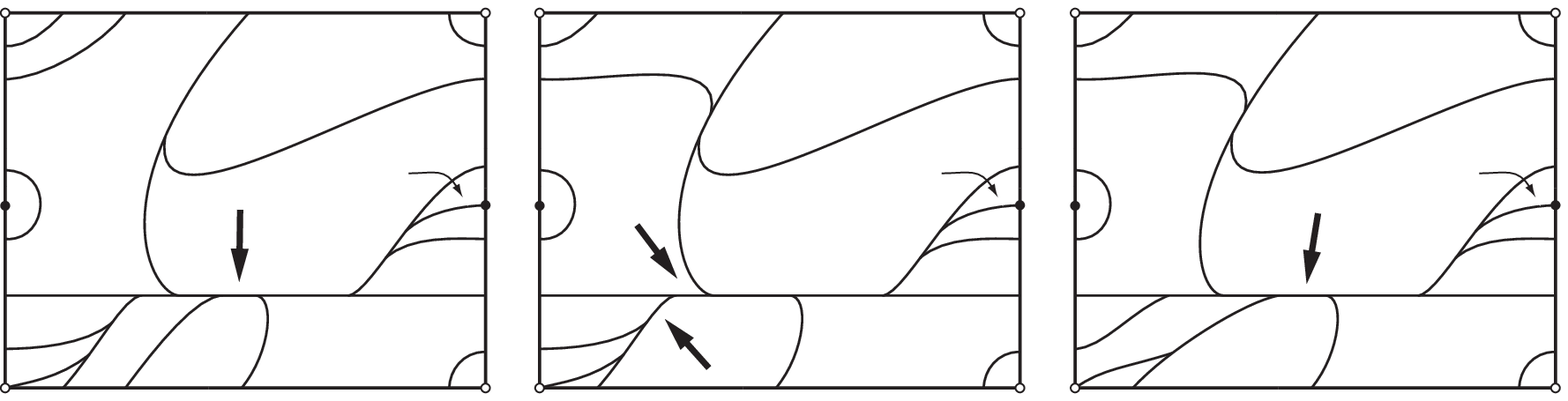} 
\caption{Sliding and folding $\theta$.\label{fig:flat}}
\end{figure} 

First, apply slide moves at the two mixed branches of $\theta$ that are adjacent to branches $b_3$ and $b_4$.
The resulting train track $\theta'$ has a single large branch, which, for the case $\rho<0$, appears in Figure \ref{fig:flat}, left. 
Suppose $v'$ is the counting measure of $t_\beta$ on $\theta'$ induced by $v$.
If $\rho\geq 0$, then $v'$ induces a measure $v''$ on the left splitting of $\theta'$ at its large branch, and the resulting train track is $\sigma_+$.
If $\rho<0$, $v'$ induces $v''$ on the right splitting, and the resulting train track is $\theta''_-$ shown in Figure \ref{fig:flat}, center.
Apply slide moves at the two branches indicated in the figure by the large arrows, so that the result is $\theta'''_-$ that is shown in Figure \ref{fig:flat}, right.
If $\rho<-r$, the measure $v'''$ on $\theta'''_-$, induced by $v''$, induces the measure $w$ on the left splitting, which results in $\sigma_{--}$. 
If $-r\leq\rho<0$, $v'''$ induces the measure $w$ on the right splitting, which results in $\sigma_-$.
\end{proof}

The Schubert train track has ten switches and over a dozen branches, yet for $t_\beta\prec\sigma$, a component of $t_\beta\setminus (\mu\cup\lambda)$ belongs to one of at most five parallel classes of subarcs in the fundamental polygon $P$.
The following lemma simplifies $\sigma$ in a way that will facilitate the parameterization of $t_\beta$ taken up in the next section.
   
\begin{figure}[ht!] 
\labellist 
\footnotesize\hair 2pt 
\pinlabel $b_1$ at 102 191
\pinlabel $b_2$ at 110 233
\pinlabel $b_3$ at 159 246
\pinlabel $b_4$ at 153 202
\pinlabel $b_5$ at 144 225
\pinlabel $b_6$ at 236 225
\pinlabel $\sigma'_{--}$ at 173 162 
\pinlabel $b_2$ at 279 191
\pinlabel $b_3$ at 284 233
\pinlabel $b_4$ at 335 246
\pinlabel $b_1$ at 328 202
\pinlabel $b_5$ at 315 231
\pinlabel $b_6$ at 411 225
\pinlabel $\sigma'_{-}$ at 344 162 
\pinlabel $b_3$ at 102 131
\pinlabel $b_2$ at 115 83
\pinlabel $b_1$ at 161 70
\pinlabel $b_4$ at 150 119
\pinlabel $b_5$ at 125 113
\pinlabel $b_6$ at 234 77
\pinlabel $\sigma'_{0}$ at 173 9
\pinlabel $b_2$ at 277 131
\pinlabel $b_1$ at 290 83
\pinlabel $b_4$ at 336 70
\pinlabel $b_3$ at 326 119
\pinlabel $b_5$ at 319 94
\pinlabel $b_6$ at 409 77
\pinlabel $\sigma'_{+}$ at 345 9
\endlabellist 
\centering 
\includegraphics[width=\textwidth]{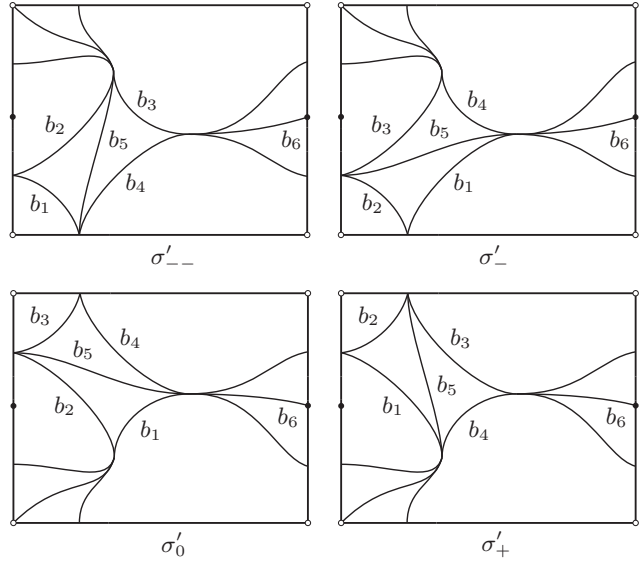} 
\caption{Folded Schubert train track $\sigma'$. \label{fig:bridgearcs}}
\end{figure} 
   
\begin{lem}\label{lem:bridgearcs}
If trivial arc $t_\beta$ of a $(1,1)$ knot with Schubert form $S(r,s,t,\rho)$ is carried by $\sigma'$, where $\sigma'$ is the train track with terminals $\sigma'_{--}$ if $\rho<-2s-r-1$, $\sigma_-'$ if $-2r-s-1\leq\rho<-r$, $\sigma_0'$ if $-r\leq\rho<t$, and $\sigma_+'$ if $\rho\geq t$, as depicted in Figure \ref{fig:bridgearcs}. 
The counting measure $w'=w'_{t_\beta}$ of $t_\beta$ on $\sigma'$ is 
$(w'_1,\dots,w'_6)=$
$$
\begin{cases}
(r+s,r,3r+s+t+1,r,-2r-s-\rho-1,1) & \text{if $\rho<-2r-s-1$,}\\
(r,-r-\rho-1,r,r+t-\rho,2r+s+\rho+1,1) & \text{if $-2r-s-1\leq\rho<-r$,}\\
(3r+s+\rho+1,r,r+\rho,r,t-\rho,1) & \text{if $-r\leq\rho<t$,}\\
(r,r+t,r,3r+s+t+1,\rho-t,1) & \text{if $\rho\geq t$.}\\
\end{cases}
$$
\end{lem}

\begin{proof}
We show that the measured train tracks $(\sigma,w)$ and $(\sigma',w')$ are equivalent.
If $\rho\geq t$, then $t_\beta\prec\sigma_+$, according to Lemma \ref{lem:traintrack}. 
Apply a sequence of slide and fold moves so as to collapse the complementary regions labeled 1 -- 6,  in that order, that appear in Figure \ref{fig:folding}. 
The resulting traintrack with induced branch weights appears in Figure \ref{fig:folding}, to the right.
Split the large branch adjacent to the two branches with weights $r+t$ and $r+\rho$. 
If $\rho\geq t$, then a left splitting of the branch results in $\sigma'_+$,
otherwise $\rho<t$ and a right splitting yields $\sigma'_0$.
The induced measure $w'$ on $\sigma'$ is easily obtained.
The other cases follow similarly, and they are omitted.
\end{proof}
 
\begin{figure}[ht!] 
\labellist 
\footnotesize\hair 2pt 
\pinlabel $1$ at 144 41
\pinlabel $2$ at 14 23
\pinlabel $3$ at 108 17
\pinlabel $4$ at 228 17
\pinlabel $5$ at 231 104
\pinlabel $6$ at 304 82
\tiny
\pinlabel $r$ at 377 61
\pinlabel $r+t$ at 382 96
\pinlabel $r+\rho$ at 400 110
\pinlabel $r$ at 442 97
\pinlabel $3r+s+\rho+1$ at 445 31
\pinlabel $3r+s+t+1$ at 450 55
\endlabellist 
\centering 
\includegraphics[width=\textwidth]{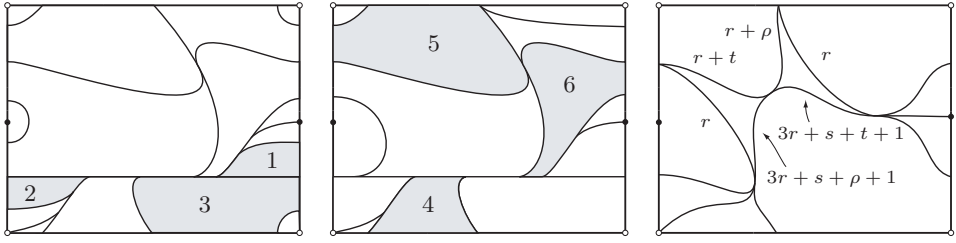} 
\caption{Folding $\sigma_+$. \label{fig:folding}}
\end{figure}

\section{Bridge sequence}

The meridians $\mu,\lambda$ divide $t_\beta$ into subarcs, which we denote by
$d_1,\dots, d_n$, each of which is properly embedded in the fundamental polygon.
Index the subarcs according to their order in $t_\beta$ from $y$ to $x$.
If $t_\beta\prec\sigma'$, then Figure \ref{fig:bridgearcs} shows the number $n$ of subarcs is $\displaystyle \Sigma_{i=1}^5 w'_i,$ where $w'$ is the counting measure of $t_\beta$ on $\sigma'$.
Parameterize $P$ by the half-open interval $[0,2n+4)\subset\mathbb{R}$, such that $p\in P$ is integral if and only if $p$ is an endpoint of $d_k$, for some $i\in\{1,\dots,n\}$ or $p$ is identified with one of the basepoints $x,y$ in the torus $\Sigma$.
Let $s_{2i-1},s_{2i}$ denote the endpoints of $d_i$, $i=1,\dots,n$. 
The ``bridge'' $t_\beta=d_1\cup\dots\cup d_n$ is uniquely determined (up to isotopy) by the sequence of endpoints $s_1,s_3,\dots,s_{2n-1}$.

\begin{defn}
The \emph{bridge sequence} of a $(1,1)$ knot in Schubert form is the sequence of integers $(s_1,s_3,\dots,s_{2i-1},\dots,s_{2n-1})$.
\end{defn}

Let $y_1,y_2,y_3,y_4\in P$ denote the points homologous to $y$, and $x_1,x_2\in P$ the points homologous to $x$, the order chosen to agree with the parameterization of $P$.
Figure \ref{fig:bridgearcs} and Lemma \ref{lem:bridgearcs} determine the integer representatives of $y_1,y_2,x_1,y_3,y_4,x_2$, and they appear in Table \ref{tbl:basepoints}. 
\begin{table}[htdp]
\caption{}\label{tbl:basepoints}
\begin{center}
\begin{tabular}{c|l|l}
& $\rho<-r$ & $\rho\geq -r$ \\
\hline 
$y_1$ & 0 & 0\\
$y_2$ & $-\rho$ & $2r+\rho+1$\\
$x_1$ & $2r+s-\rho+1$ & $4r+s+\rho+2$\\
$y_3$ & $4r+s+t-\rho+2$ & $6r+s+t+\rho+3$\\
$y_4$ & $4r+s+t-2\rho+2$ & $8r+s+t+2\rho+4$\\
$x_2$ & $6r+s+2t-2\rho+3$ & $10r+s+2t+2\rho+5$
\end{tabular}
\end{center}
\end{table}
The initial and final points of $t_\beta$ follow similarly: 
\begin{eqnarray}
s_1&=&\left\{
			\begin{array}{ll}
				y_3 & \mbox{if }\rho<-r,\\
				y_2 & \mbox{if }\rho\geq -r,
			\end{array}\right.\\
s_{2n}&=&x_2.
\end{eqnarray}
The branches $b_1,\dots,b_4$ of $\sigma'$ are ordered so as to facilitate the statement of the following Proposition \ref{prop:bridge}, which computes the bridge sequence explicitely in terms of Schubert form.
Define the \emph{center} $p_i$ of branch $b_i$, $i=1,\dots,4$ to be as in Table \ref{tbl:bridgearcs}; that is,
$p_i$ is a point in $\{y_1,y_2,x_1,y_3,y_4,x_2\}-\{s_1,s_{2n}\}$ that is equidistant in $P$ from the ends of each subarc carried by $b_i$.
\begin{table}[htdp]
\caption{}
\begin{center}
\begin{tabular}{c|c|c|c|c}
& $\rho<-2r-s-1$ & $-2r-s-1\leq\rho<-r$ & $-r\leq\rho<t$ & $t\leq\rho$ \\
\hline
$p_1$ & $y_2$ & $y_1$ & $y_1$ & $x_1$ \\
$p_2$ & $x_1$ & $y_2$ & $x_1$ & $y_3$ \\
$p_3$ & $y_4$ & $x_1$ & $y_3$ & $y_4$ \\
$p_4$ & $y_1$ & $y_4$ & $y_4$ & $y_1$ \\  
\end{tabular}\label{tbl:bridgearcs}
\end{center}
\end{table}

\begin{prop}\label{prop:bridge}
Let $S=S(r,s,t,\rho)$ be the Schubert form of a $(1,1)$ knot. 
If $w'$ is the counting measure of $t_\beta$ on $\sigma'$, then the terms of the bridge sequence of $S$ are
$$
s_{2i-1}=(\psi\phi)^{i-1}(s_1),\quad i=1,\dots,n
$$
where $\psi,\phi$ are permutations given by the cycle notation below:
$$
\phi = 
\prod_{i=1}^4\left(\prod_{j=1}^{w'_i}(\,\overline{p_i-j}\;\;\;\overline{p_i+j}\,)\right)\prod_{k=1}^{w'_5}(\,
\overline{p_1-w'_1-k}\;\;\;\overline{p_2+w'_2+k}\,), 
$$
$$\psi= \prod_{i=1}^{y_2-y_1-1}(\,y_1+i\;\;\;y_4-i\,)\prod_{j=1}^{y_3-y_2-1}(\,y_4+j\;\;\;y_3-j\,),
$$
and $\overline{a}=a \mod{2n+4}$.
\end{prop}

Note that the initial term $s_1$ from which the proposition computes the bridge sequence is given by (1) and Table \ref{tbl:basepoints}, hence, the bridge sequence is completely specified by the proposition.
  
\begin{proof}
The $i$th term $s_{2i-1}$ of the bridge sequence corresponds to the initial endpoint of the subarc $d_i\subset t_\beta$.
For $i<n$, the next term, $s_{2i+1}$, in the bridge sequence is the (unique) point in $P$ that is homologous (in $\Sigma$) to the final endpoint $s_{2i}$ of $d_i$. 
We will show that $\phi$ transposes the endpoints of $d_i$, and $\psi$ transposes homologous (integer) pairs in $P$, hence 
$$
s_{2i+1}=\psi\phi(s_{2i-1}),
$$
and the bridge sequence is obtained from $s_1$ by iterating $\psi\phi$ as claimed.

Consider $\phi$. 
There are five parallel classes of arcs among $d_1,\dots,d_n$, and they correspond to the branches $b_1,\dots,b_5$ of the train track $\sigma'$. 
For $i=1,\dots,4$, the branch $b_i$ is centered on $p_i$, and the endpoints of an arc carried by $b_i$ lie at a distance $j=1,\dots,w'_i$ to either side of $p_i$.
Given that $P=[0,2n+4)$, the two points at distance $j$ from $p_i$ are $\overline{p_i\pm j}$, where $\overline{a}=a \mod{2n+4}$.
Hence, the $2$-cycle
$(\,\overline{p_i-j}\;\;\;\overline{p_i+j}\,)$ transposes endpoints of arcs carried by the branch $b_i$, $i=1,\dots,4$.
As for $b_5$, one end is adjacent to the end of $b_1$ preceeding $p_1$ (in $P$), and the other is adjacent the end of $b_2$ succeeding $p_2$.
Therefore, one endpoint of an arc carried by $b_5$ is less than $p_1$ by $w'_1-k$, and the other is greater than $p_2$ by $w'_2+k$ for some  $k=1,\dots,w'_5$; 
the permutation $(\overline{p_1-w'_1-k}\;\;\;\overline{p_2+w'_2+k})$ transposes the two.
Multiplying all the above transpositions results in the permutation $\phi$, as required.

The ``gluing'' permutation $\psi$ is also a product of disjoint transpositions.
The two sequences of integers $y_1+1,y_1+2,\dots,y_2-1$ and $y_3+1,y_3+2,\dots,y_4-1$ are identified in $\lambda\subset\Sigma$, in reverse order, and the homologous pairs are transposed by cycles
$(\,y_1+i\;\;\;y_4-i\,)$, for $i=1,\dots,y_2-y_1-1$.
Similarly, the integer sequences $y_2+1,y_2+2,\dots,y_3-1$ and $y_4+1,y_4+2,\dots,\overline{y_1-1}$ are identified in $\mu\subset\Sigma$, in reverse order.
The corresponding transpositions are $(\,y_2+j\;\;\;\overline{y_1-j}\,)$, for $j=1,\dots,y_3-y_2-1$.
(There is no single point identified with one of $y_1,\dots,y_4$; they have been excluded from the permutation $\psi$ because they are already all identified with $s_1$.)
The product of the transpositions associated with $\lambda$ and $\mu$ is $\psi$, as claimed.
\end{proof}

\begin{exmp}[Trefoil bridge sequence]\label{ex:trefoilbridge}
A Schubert form for the trefoil is $(r,s,t,\rho)=(0,0,2,2)$. 
From Lemma \ref{lem:bridgearcs}, we have $(w'_1,\dots,w'_5)=(0,2,0,3,0)$, with ``centers'' given by Table \ref{tbl:bridgearcs}: $(p_1,p_2,p_3,p_4)=(x_1,y_3,y_4,y_1)$.
From Table \ref{tbl:basepoints}, basepoints $y_1,y_2,x_1,y_3,y_4,x_2$ take positions $0,3,4,7,10,13$, respectively, hence $(p_1,p_2,p_3,p_4)= (4,7,10,0)$.
The gluing permutation is
\begin{eqnarray*}
\psi &=& \prod_{i=1}^2(\,0+i\;\;\;10-i\,)\prod_{j=1}^3(\,10+j\;\;\;7-j\,)\\
&=& (1\;\;9)(2\;\;8)(11\;\;6)(12\;\;5)(13\;\;4),
\end{eqnarray*}
where $\overline{a}=a\mod{2n+4}=a\mod{14}$, since $n=\Sigma_{i=1}^5 w'_i =5$.
There are only two branches of positive weight, $b_2$ and $b_4$, and the subarc permutation is
\begin{eqnarray*}
\phi &=& 
\prod_{j=1}^{2}(\,\overline{7-j}\;\;\;\overline{7+j}\,)
\prod_{j=1}^{3}(\,\overline{0-j}\;\;\;\overline{0+j}\,)
\\
&=& (6\;\;8)(5\;\;9)(13\;\;1)(12\;\;2)(11\;\;3).
\end{eqnarray*}

Since $\rho\geq -r$, the initial term $s_1$ is $y_2=2r+\rho+1=3$. 
Going to the end of the first subarc, we meet $\phi(3)=11$, which is homologous to $s_3=\psi(11)=6$.
Thus, the bridge sequence begins $3,6$. 
Continuing in this way, alternating $\phi$ and $\psi$ a total of $n=5$ times, we obtain the bridge sequence $(3,6,2,5,1)$.
See Figure \ref{fig:trefoilsequence}.
\end{exmp}

\section{Attaching sequence}

Let $t_\beta$ be the bridge of a $(1,1)$ knot in Schubert form $S$, whose bridge sequence is $(s_1,s_3,\dots,s_{2n-1})$.
By $t_\beta^i$, we denote the $i$th partial bridge $d_1\cup\dots\cup d_i$, for $i=1,\dots,n$. 
Define $t_\beta^0:=s_1$. 
The same parameterization of $P$ developed in the previous section to characterize $t_\beta$ also serves to characterize an oriented attaching circle $\gamma\subset\Sigma-t_\beta^i$.

\begin{defn}
Let $i$ be an integer such that $0\leq i\leq n$.
An \emph{attaching sequence for $t_\beta^i$} is a sequence of half-integer points $(c_1,\dots,c_m)$ in $P$ for which there is an oriented simple closed curve $\gamma\subset\Sigma-t_\beta^i$ with the property that $\gamma\setminus P$ consists of $m$ arcs, and the tail endpoint $c_j'$ of the $j$th arc satisfies the inequality $c_j-\frac{1}{2}<c_j'<c_j+\frac{1}{2}$ for all $j=1,\dots,m$.
\end{defn}

An attaching sequence $(c_1,\dots,c_m)$ determines the curve $\gamma$ uniquely (up to isotopy), and we will therefore use the attaching sequence to represent the curve.
Extend the gluing map $\psi$ defined in Proposition \ref{prop:bridge} to the half integers:
$$\psi= \prod_{i=1}^{2y_2-2y_1-1}\left(\,y_1+\frac{i}{2}\;\;\;y_4-\frac{i}{2}\,\right)\prod_{j=1}^{2y_3-2y_2-1}\left(\,y_4+\frac{j}{2}\;\;\;y_3-\frac{j}{2}\,\right).
$$
Finally, we use
interval notation $(a,b)$ to denote the clockwise-oriented segment in $P$, from $a$ to $b$;
$$(a,b)=\left\{\begin{array}{ll}
\{x\in P :a<x<b\} & \mbox{ if } a\leq b,\\ 
\{x\in P :b<x \mbox{ or } x<a\} & \mbox{ if } b<a.
\end{array}\right.
$$

\begin{prop}\label{prop:detour}
Let $(s_1,s_3,\dots,s_{2n-1})$ be the bridge sequence of a $(1,1)$ knot in Schubert form $S$.
Suppose $\gamma=(c_1,c_2,\dots, c_m)$ is an attaching sequence for $t_\beta^i$. 
Let $a=s_{2i-1}$ and $b=\phi(s_{2i-1})$.
Define
$$
f_{S,i}(\gamma)=(f_{S,i}(c_1), \dots, f_{S,i}(c_j), \dots, f_{S,i}(c_m)),
$$
where
\begin{displaymath}
f_{S,i}(c_j)=\left\{
	\begin{array}{ll}
		c_j, \psi(b-\frac{1}{2}), b+\frac{1}{2}, & \mbox{if }c_j\in(a,b) \mbox{ and } \psi(c_{\overline{j+1}})\in(b,a),\\
		c_j, \psi(b+\frac{1}{2}), b-\frac{1}{2}, & \mbox{if }c_j\in(b,a) \mbox{ and } \psi(c_{\overline{j+1}})\in(a,b),\\
		c_j, & \mbox{else,}\\
	\end{array}\right.
\end{displaymath}
for $\overline{j}=j\mod{m}$, and $i=1,\dots,n$. Then $f_{S,i}(\gamma)$ is an attaching sequence for $t_\beta^{i+1}$, and the curve represented by $f_{S,i}(\gamma)$ is isotopic to $\gamma$.

\end{prop}

\begin{proof}
We will apply a ``finger move'' along the subarc of $t_\beta$ from $a$ to $b$, and then observe that the result has attaching sequence $f_{S,i}(\gamma)$, as defined.

Let $g_1,\dots,g_p\subset \gamma$ be arcs, each of which intersects $P$ exactly in its two endpoints, one of lies in $(a,b)$ and the other in $(b,a)$.
Cut the fundamental polygon along $t_\beta^i$ and $\gamma-g_1-\dots-g_p$.
Among the regions remaining, there is one, $R$, that contains $a,b$.
Take a simple path in the boundary of $R$ that connects $a$ and $b$, and push it just off the boundary to obtain an arc $u$ isotopic to $d_{i+1}=t_\beta^{i+1}-t_\beta^{i}$. 
The arc $u$ is disjoint from $t_\beta^i$ and intersects $\gamma$ in $p$ points, with one intersection point in each of $g_1,\dots,g_p$. 
Relabeling if necessary, index $g_1,\dots,g_p$ by the order in which $u$ meets them, from $b$ back to $a$.
\begin{figure}[ht!] 
\labellist 
\footnotesize\hair 2pt 
\pinlabel $s_{2i+1}$ at 97 43
\pinlabel $s_{2i+2}$ at 430 111
\pinlabel $U_k$ at 162 65
\pinlabel $g'_k$ at 243 119
\pinlabel $u$ at 295 85
\endlabellist 
\centering 
\includegraphics[width=\textwidth]{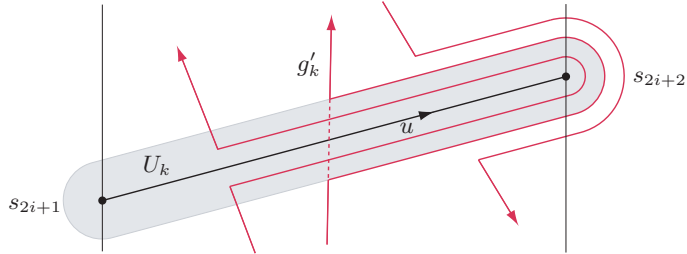} 
\caption{Finger move.\label{fig:finger}}
\end{figure}

For each $k=1,\dots,p$, take an $\frac{\epsilon}{k}$-radius tube $U_k$ around $u$, where $0<\epsilon\ll 1$. 
The intersection $g_k\cap U_k$ divides $U_k$ into two simply connected regions, of which one, call it $D_k$, is disjoint from $g_{k+1},\dots,g_p$ and contains $b$.
The boundary of $D_k$ consists of $g_k\cap U_k$ joined with the path $\partial U_k\cap D_k$.
Let $g_k'$ be the path obtained by replacing the subarc $(g_k\cap U_k)$ of $g_k$ by $\partial U_k\cap D_k$. 
The curve $\gamma'$ obtained by replacing $g_k$ by $g_k'$ in $\gamma$ for all $k=1,\dots,p$, is isotopic to $\gamma$ because $g_k$ and $g_k'$ are isotopic.
By construction, $\gamma'$ is disjoint from the partial bridge $t_\beta^{i+1}=t_\beta^{i}\cup u$.
See Figure \ref{fig:finger}.

An arc $g_k$ that crosses $u$ from left to right corresponds to $c_j\in(a,b)$ for which $\psi(c_{\overline{j+1}})\in(b,a)$.
In this case, $g_k'$ passes around $u$ (at a distance of $\epsilon/k$) from left to right; $g_k'$ is comprised of three arcs whose endpoints are
$
c_j,w-\frac{1}{2}
$;\quad 
$\psi(b-\frac{1}{2}),\psi(b+\frac{1}{2});
$\quad
and
$
b+\frac{1}{2},\psi(c_{\overline{j+1}}).
$
Replacing $g_k$ by $g_k'$, takes $c_j$ to the three initial points of the above arcs, as required by the first case in the definition of $f_{S,i}$. 
If $g_k$ crosses $u$ from right to left, then the second case results. 
Arcs of $\gamma$ that do not cross $u$ are unchanged by $f_{S,i}$ and this is the content of the third and final case.
Having demonstrated that in all cases the term $c_j$ of $\gamma$ is replaced by $f_{S,i}(c_j)$ in $\gamma$, the proof is complete.  
\end{proof}

\begin{theorem}\label{thm:construction}
Let $K$ be a $(1,1)$ knot with $(1,1)$ decomposition $(H_\alpha,t_\alpha)\cup_h(H_\beta,t_\beta)$ and Schubert form $S=S(r,s,t,\rho)$.
Let $\Sigma=\partial H_\alpha$ and $\{x,y\}=\partial t_\alpha=\partial t_\beta$.
Suppose $s_1,s_3,\dots, s_{2n-1}$ is the bridge sequence of $S$.
Let
\begin{eqnarray*}
\alpha&=&\left\{
	\begin{array}{ll}
		-\rho -\frac{1}{2} & \mbox{if }\rho<-r,\\
		2r+\rho+\frac{1}{2} & \mbox{if }\rho\geq -r.
				\end{array}\right.
\end{eqnarray*}
Define $\beta=\beta_n$, where
\begin{eqnarray*}
\beta_0&=&\left\{
	\begin{array}{ll}
		-\rho +\frac{1}{2} & \mbox{if }\rho<-r,\\
		2r+\rho+\frac{3}{2} & \mbox{if }\rho\geq -r,
				\end{array}\right.\\
\beta_{i+1}&=&f_{S,i}(\beta_i),
\end{eqnarray*}
for $i=0,\dots,n-1$.
Then $(\Sigma,\alpha,\beta,x,y)$ is a genus 1 doubly-pointed Heegaard diagram compatible with $K$.
\end{theorem}

\begin{proof}
As defined, $\alpha$ and $\beta_0$ correspond to the half-integers $y_2-\frac{1}{2}$ and $y_2+\frac{1}{2}$, which belong to $\lambda$ and $\mu$, respectively. 
As (one-term) attaching sequences, they represent simple closed curves parallel to $\mu$ and $\lambda$; that is, $y_2-\frac{1}{2}$ and $y_2+\frac{1}{2}$ are meridian curves for $H_\alpha$ and $H_\beta$, respectively. 
Thus, $(\Sigma,\alpha,\beta_0)$ is a Heegaard diagram for the $S^3$.
By Proposition \ref{prop:detour}, $f_{S,i}(\beta_i)$ is isotopic to $\beta_i$ for all $i=0,\dots,n$.
Therefore $\beta$ is isotopic to $\beta_0$, and $(\Sigma,\alpha,\beta)$ is also a Heegaard diagram for $S^3$.

According to Schubert form, the trivial arc $t_\alpha$ belongs to the meridian disk bounded by $\mu$.  
Since $\mu$ and $\alpha$ are parallel in $\Sigma$, we can define $D_\alpha\subset H_\alpha$ to be the parallel meridian disk with $\partial D_\alpha=\alpha$, for which $D_\alpha\cap t_\alpha=\emptyset$.
From Proposition \ref{prop:detour} it follows that the curve $\beta_i$ is disjoint from $t_\beta^i$ for all $i=0,\dots,n$. 
Thus $\beta=\beta^n$ is disjoint from $t_\beta=t_\beta^n$, and any properly embedded disk $D_\beta\subset H_\alpha$ with $\partial D_\beta=\beta$ is disjoint from $t_\beta$.
In conclusion, the unique isotopy classes of trivial arcs between $x$ and $y$ in $H_\alpha-D_\alpha$ and $H_\beta-D_\beta$ correspond to $t_\alpha$ and $t_\beta$, respectively, and $(\Sigma,\alpha,\beta,x,y)$ is a doubly-pointed Heegaard diagram compatible with $K=t_\alpha\cup t_\beta$.
\end{proof}

\begin{defn}
An attaching sequence $c_1,\dots,c_m$ for $t_\beta^i$ is \emph{reduced} if for all $j=1,\dots,m$, $(c_j,\psi(c_{\overline{j+1}}))\cap t_\beta\neq\emptyset$ and $(\psi(c_{\overline{j+1}}),c_j)\cap t_\beta\neq\emptyset$, where $\overline{j}=j \mod{m}$.
A pair of terms $c_j,c_{\overline{j+1}}$ that does not meet the above conditions is called an \emph{inessential pair}. 
\end{defn}

\begin{lem}\label{lem:inessentialpairs}
Suppose $\gamma$ is an attaching sequence for
$t_\beta^i$.
If $\gamma'$ is the sequence obtained by removing every inessential pair from $\gamma $, then $\gamma'$ is also an attaching sequence for $t_\beta^i$, and $\gamma $ and $\gamma'$ are isotopic as curves. 
\end{lem}

\begin{proof}
Let $c_1,\dots,c_m$ be the terms of the attaching sequence $\gamma$.
Suppose $c_j,c_{\overline{j+1}}$ is an inessential pair, for which, say, $(c_j,\psi(c_{\overline{j+1}}))\cap t_\beta^i=\emptyset$. 
Then $d_j\cup(c_j,\psi(c_{\overline{j+1}}))$ bounds a disk in $\Sigma-t_\beta^i$, where $d_j\subset \gamma$ is the properly embedded subarc with endpoints $c_j,\psi(c_{\overline{j+1}})$. 
Among all such disks there is an innermost disk whose interior is disjoint from $\gamma$. 
Cancel the inessential arc by an isotopy taking it across the disk, thereby removing the points of intersection corresponding to the inessential pair.
Let $\gamma'$ be the result of having applied the same isotopy to the next innermost disk, and so on, until all inessential arcs are removed.
Then $\gamma'$ is represented by attaching sequence $c'$ and $\gamma'$ is isotopic to $\gamma$.
\end{proof}

From now on we assume that an attaching sequence is reduced.
For otherwise we may take any given attaching sequence and successively remove inessential pairs until it is reduced.
By Lemma \ref{lem:inessentialpairs}, the result is an attaching sequence for the same curve.

\begin{exmp}[Trefoil attaching sequence]\label{ex:trefoilsequence}

Schubert form, bridge sequence, and the permutations $\phi$ and $\psi$ were obtained in Example \ref{ex:trefoilbridge} as
\begin{eqnarray*}
S&=&S(0,0,2,2),\\
s_1,s_3,\dots,s_{2n-1}&=&3,6,2,5,1,\\
\phi&=&(6\;\;8)(5\;\;9)(13\;\;1)(12\;\;2)(11\;\;3),\\ 
\psi&=&(1\;\;9)(2\;\;8)(11\;\;6)(12\;\;5)(13\;\;4).
\end{eqnarray*}
The extended gluing permutation is 
\begin{eqnarray*}
\psi &=& (0.5\;\;9.5)(1\;\;9)(1.5\;\;8.5)(2\;\;8)(2.5\;\;7.5)(10.5\;\;6.5)\\
&&\quad(11\;\;6)(11.5\;\;5.5)(12\;\;5)(12.5\;\;4.5)(13\;\;4)(13.5\;\;3.5).
\end{eqnarray*}
Given that $\rho\geq -r$, we have
\begin{eqnarray*}
\alpha=2.5, &&\beta_0=3.5.
\end{eqnarray*}
Therefore $\alpha$ meets $P$ at $2.5$ and $\psi(2.5)=7.5$. 
To determine $\beta_1=f_{S,1}(\beta_0)$, note that the first bridge arc has endpoints $a=s_1=3$ and $b=\phi(s_1)=11$. 
Since $\beta_0$ meets $P$ at $3.5$ and $\psi(3.5)=13.5$, which belong to $(a,b)$ and $(b,a)$, respectively, $f_{S,1}(\beta_0)=f_{S,1}(3.5)=(3.5,\psi(w-.5),w+.5)=(3.5,6.5,11.5)$.
However, the pair $11.5,3.5$ is inessential since the interval $(11.5,\psi(3.5))=(11.5,13.5)$ is disjoint from $t_\beta^1$, which meets $P$ in the points $\{0,3,7,10,11\}$.
The reduced attaching sequence of $t_\beta^1$ is
$$
\beta_1=6.5.
$$
For the next bridge arc, $a=6$ and $b=8$.
Here $f_{S,2}(\beta_1)=f_{S,2}(6.5)=(6.5,2.5,8.5)$, since $6.5\in(a,b)$ and $\psi(6.5)=10.5\in(b,a)$.
This attaching sequence is reduced; 
$$
\beta_2=6.5,2.5,8.5.
$$
For $i=3$, $a=2$ and $b=12$. 
Here $f_{S,3}(6.5)=6.5$, $f_{S,3}(2.5)=(2.5,5.5,12.5)$, and $f_{S,3}(8.5)=8.5$.
The sequence is reduced;
$$
\beta_3=6.5,2.5,5.5,12.5,8.5.
$$
The computations for $i=4,5$ follow similarly, and the results are
\begin{eqnarray*}
\beta_4&=&6.5, 2.5, 5.5, 1.5, 9.5, 12.5, 9.5,\\
\beta_5&=&6.5, 2.5, 5.5, 1.5, 4.5, 13.5, 9.5, 13.5, 9.5.
\end{eqnarray*} 
See Figure \ref{fig:trefoilsequence}.
\end{exmp}

\begin{figure}[ht!] 
\labellist 
\footnotesize\hair 2pt 
\pinlabel $0$ at 132 214
\pinlabel $3$ at 2 214
\pinlabel $7$ at 1 385
\pinlabel $10$ at 136 386
\pinlabel $t_\beta^1$ at 265 276
\pinlabel $t_\beta^2$ at 450 276
\pinlabel $t_\beta^3$ at 68 92
\pinlabel $t_\beta^4$ at 251 92
\pinlabel $t_\beta$ at 438 90
\color{blue}
\pinlabel $\alpha$ at 40 303
\pinlabel $\alpha$ at 403 100
\color{red}
\pinlabel $\beta_0$ at 87 257
\pinlabel $\beta_1$ at 269 345
\pinlabel $\beta_2$ at 448 345
\pinlabel $\beta_3$ at 85 137
\pinlabel $\beta_4$ at 268 135
\pinlabel $\beta$ at 457 139
\endlabellist 
\centering 
\includegraphics[width=\textwidth]{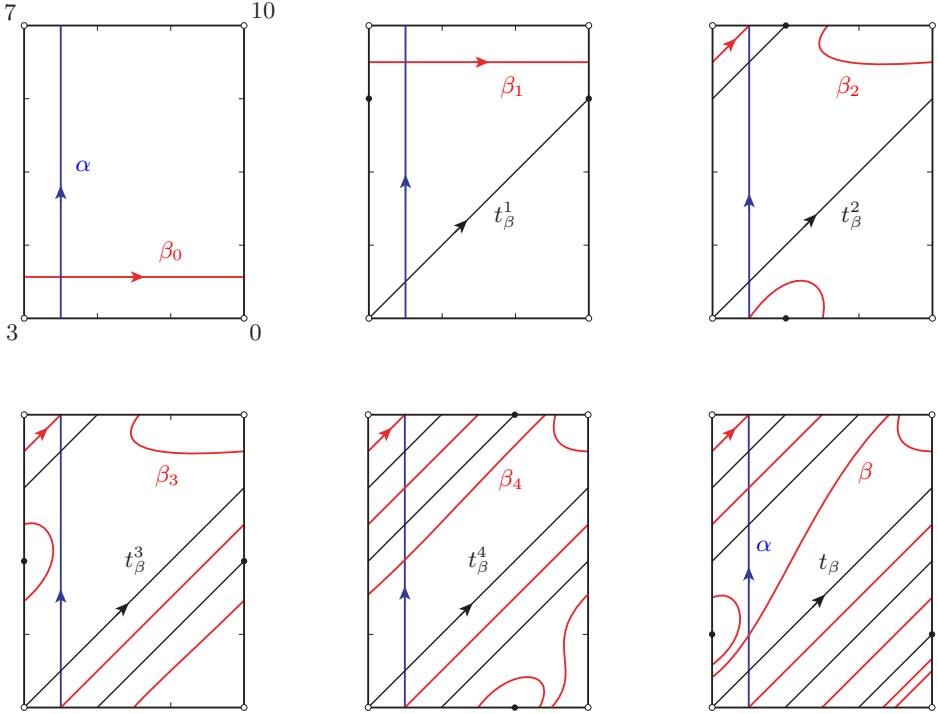} 
\caption{Attaching circle construction for the trefoil.\label{fig:trefoilsequence}}
\end{figure}

\section{Heegaard diagram Schubert form}

\begin{defn}
Let $\sigma\subset\hat{\Sigma}$ be the Schubert train track, and define $\alpha\subset\hat{\Sigma}\setminus\mu$ to be a meridian parallel to $\mu$.
Suppose $w=w_\beta$ is the counting measure of a curve $\beta\prec\sigma$, and $(w_1,\dots,w_5)=(r',s',t',\rho',0)$. 
If $(\Sigma,\alpha,\beta,x,y)$ is a doubly-pointed Heegaard diagram compatible with a knot $K$, then we say $HS(r',s',t',\epsilon\rho')$ is a \emph{Heegaard diagram Schubert form} of $K$, where the sign $\epsilon=\pm 1$ is positive if $\sigma$ equals $\sigma_+$, and negative otherwise.
\end{defn}

We note that Cattabriga and Mulazzani defined an integer 4-tuple parameterization $K(a,b,c,r)$ of (1,1) knots in terms of Heegaard diagrams and Dunwoody manifolds \cite{Cattabriga-Mulazzani}.
Inspection reveals that the Heegaard diagram Schubert form $HS(r',s',t',\rho')$ corresponds to 
$K(r',s',t',\rho'-s')$.

\begin{defn}
The \emph{intersection numbers} of an attaching sequence $\gamma=c_1,\dots,c_m$ is a 4-tuple of non-negative integers $(a,b,c,d)$, where
\begin{eqnarray*}
a&=&|\{i:y_2\leq c_i<y_3\mbox{ and }y_4\leq c_{\overline{i+1}}<2y_2\}|,\\
b&=&|\{i:y_2\leq c_i <x_1 \mbox{ or } x_2\leq c_i <2y_2\}|,\\
c&=&|\{i:x_1\leq c_i <y_3 \mbox{ or } y_4\leq c_i <x_2\}|,\\
d&=&|\{i:y_1\leq c_i <y_2 \mbox{ or } y_3\leq c_i <y_4\}|,
\end{eqnarray*}
and $\overline{j}=j\mod{m}$. 
\end{defn}

Intersection number $a$ counts the number of subarcs of $\gamma$ that intersect $P$ exactly in their endpoints, both of which belong to $(y_2,y_3)$
And $b$, $c$, $d$ represent the geometric intersection numbers $| \gamma\cap(y_2,x_1)|$, $| \gamma\cap(x_1,y_3)|$, $|\gamma\cap(y_1,y_2)|$, respectively.

\begin{theorem}\label{thm:heegform}
If $(a,b,c,d)$ are intersection numbers of a reduced attaching sequence $\beta$ for the bridge $t_\beta$ of a nontrivial $(1,1)$ knot $K$ in Schubert form $S(r,s,t,\rho)$, then a Heegaard diagram Schubert form $HS(r',s',t',\rho')$ of $K$ is provided by one of the following cases:
\begin{enumerate}[i.]
	\item If $s\neq 0 \neq t$, or $2a\leq b$ and $2a\leq c$, then
	$$
	(r',s',t',\rho')=
	\begin{cases}
	(a,-2a+b,-2a+c,-d) & \text{if $\rho<-r$,}\\
	(a,-2a+b,-2a+c,-2a+d) & \text{if $\rho\geq -r$.}
	\end{cases}
	$$
		
	\item If $s=0<t$ and $2a>b$, then
	$$	
	(r',s',t',\rho')=
	\begin{cases}
	(-a+b,-2a+c,2a-b,2a-2b-d) & \text{if $\rho< -r$,}\\
	(-a+b,-2a+c,2a-b,-2b+d) & \text{if $\rho\geq -r$.}
	\end{cases}
	$$
	
	\item If $s>0=t$ and $2a>c$, then
	$$	
	(r',s',t',\rho')=
	\begin{cases}
	(-a+c,2a-c,-2a+b,a-c-d) & \text{if $\rho< -r$,}\\	(-a+c,2a-c,-2a+b,d) & \text{if $\rho\geq -r$.}
	
	\end{cases}
	$$
	
	\item If $s=0=t$ and $2a>c\geq b$, then $(r',s',t',\rho')=$
	$$
	\begin{cases}
		(-a+c,b-c-\overline{b-2a},\overline{b-2a},2a-c-d-\overline{b-2a}) & \text{if $\rho< -r$,}\\
	(-a+c,\overline{2a-c},b-c-\overline{2a-c}, -b+d+\overline{2a-c}) & \text{if $\rho\geq -r$,}
	\end{cases}
	$$
	where $\overline{x}=x\mod{b-c}$.

	\item If $s=0=t$ and $2a>b>c$, then $(r',s',t',\rho')=$
	$$
	\begin{cases}
	(-a+b,-b+c-\overline{2a-b},\overline{2a-b},2a-2b+c-d-\overline{2a-b}) & \text{if $\rho< -r$,}\\
	(-a+b,-b+c-\overline{2a-b},\overline{2a-b},-b+d+\overline{2a-b}) & \text{if $-r\leq\rho<0$,}\\
		(-a+b,\overline{-2a+b},-b+c-\overline{-2a+b},-b+d+\overline{-2a+b}) & \text{if $\rho>0$,}
	\end{cases}
	$$
	where $\overline{x}=x\mod{-b+c}$.
\end{enumerate}
\end{theorem}

\begin{proof}

\emph{Case i}. According to Lemma \ref{lem:extensions}, $\beta\prec\sigma$.
The only branch arc of $\sigma$ with boundary in $(y_2,y_3)$ in the fundamental polygon is $b_1$ (see Figure \ref{fig:traintrack}), hence
\begin{equation}
a=w_1=r',
\end{equation}
where $w=w_\beta$ is the counting measure of $\beta$ on $\sigma$.
Intersection numbers $b$, $c$, and $d$ equal the sum of $w$ over the branches of $\sigma$ that are transverse to the intervals $(y_2,x_1)$, $(x_1,y_3)$, and $(y_1,y_2)$, respectively.
Using the switch conditions on $\sigma$, which determine the value of $w$ on $\mathcal{B}(\sigma)$  in terms of $w_1,\dots,w_5$, and the definition of the Heegaard diagram Schubert parameters, we deduce:
\begin{eqnarray}
b&=&2w_1+w_2=2r'+s',\\
c&=&2w_1+w_3=2r'+t',\\
d&=&
\begin{cases}
2w_1+w_4=2r'+\rho' & \text{if $\sigma=\sigma_+$,}\\
(w_1-w_4)+w_1=2r'+\rho' & \text{if $\sigma=\sigma_-$,}\\
(w_4-w_1-w_5)+w_1=-\rho' & \text{if $\sigma=\sigma_{--}$.}
\end{cases}
\end{eqnarray}
We obtain the required $(r',s',t',\rho')$ as the solution to the system of linear equations $(3)$-$(6)$.

\emph{Case ii.} Lemma \ref{lem:extensions} implies that $\beta\prec\tau_2$, where $\tau_2$ is defined by Figure \ref{fig:extensions}. 
The intersection numbers correspond to a system of linear equations in terms of $v_1,\dots,v_5$, where $v=v_\beta$ is the counting measure of $\beta$ on $\tau_2$, 
\begin{eqnarray*}
a&=&v_1+v_2,\\
b&=&2v_1+v_2,\\
c&=&2v_1+2v_2+v_3,\\
d&=&
\begin{cases}
v_2+v_4-v_5 & \text{if $\rho<-r$,}\\
2v_1+v_2-v_4 & \text{if $-r\leq\rho<0$,}\\
2v_1+2v_2+v_4 & \text{if $\rho>0$,}
\end{cases}
\end{eqnarray*}
for which the solution is 
$$
(v_1,\dots,v_4)=
\begin{cases}
(-a+b,2a-b,-2a+c,-2a+b+d) & \text{if $\rho<-r$,}\\
(-a+b,2a-b,-2a+c,b-d) & \text{if $-r\leq\rho<0$,}\\
(-a+b,2a-b,-2a+c,-2a+d) & \text{if $\rho>0$.}
\end{cases}
$$
Lemma \ref{lem:heegform} determines the required Heegaard diagram Schubert form from the measure $v$.  

\emph{Cases iii-v.} These follow from computations similar to \emph{Case ii}, and they are omitted.
\end{proof}
 
\begin{lem}\label{lem:extensions}
Let $\tau_i^{--}$, $\tau_i^-$, $\tau_i^+$ denote the Schubert train tracks $\sigma_{--}$, $\sigma_-$, $\sigma_+$, respectively for $i=1$ and the train tracks depicted in Figure \ref{fig:extensions} for $i=2,3,4,5$.
If $(a,b,c,d)$ are intersection numbers of a reduced attaching sequence $\beta$ for the bridge $t_\beta$ of a non-trivial $(1,1)$ knot $K$ in Schubert form $S(r,s,t,\rho)$, then $\beta\prec\tau$, where

$$
\tau=
\begin{cases}
\tau_1 & \text{if $s>0$ and $t>0$, or $2a\leq b$ and $2a\leq c$,}\\
\tau_2 & \text{if $s=0<t$ and $2a>b$,}\\
\tau_3 & \text{if $s>0=t$ and $2a>c$,}\\
\tau_4 & \text{if $s=0=t$ and $2a>c\geq b$,}\\
\tau_5 & \text{if $s=0=t$ and $2a>b>c$,}\\
	\end{cases}
$$
and $\tau_i$ is the train track $\tau_i^{--}$ if $\rho<-r$, $\tau_i^-$ if $-r\leq\rho<0$, and $\tau_i^+$ if $\rho>0$ for $i=1,\dots,5$. 
\end{lem}

\begin{figure}[ht!] 
\labellist 
\footnotesize\hair 2pt 
\pinlabel $b_1$ at 18 565
\pinlabel $b_2$ at 115 529
\pinlabel $b_3$ at 94 596
\pinlabel $b_4$ at 49 558
\pinlabel $b_5$ at 129 578
\pinlabel $\tau^{--}_2$ at 88 507 
\pinlabel $b_1$ at 195 570
\pinlabel $b_2$ at 290 529
\pinlabel $b_3$ at 268 599
\pinlabel $b_4$ at 234 628
\pinlabel $b_5$ at 304 580
\pinlabel $\tau^{-}_2$ at 259 507 
\pinlabel $b_1$ at 370 570
\pinlabel $b_2$ at 465 529
\pinlabel $b_3$ at 414 606
\pinlabel $b_4$ at 454 576
\pinlabel $b_5$ at 475 585
\pinlabel $\tau^{+}_2$ at 434 507
\pinlabel $b_1$ at 18 428
\pinlabel $b_2$ at 45 401
\pinlabel $b_3$ at 119 455
\pinlabel $b_4$ at 63 417
\pinlabel $b_5$ at 122 424
\pinlabel $\tau^{--}_3$ at 88 345
\pinlabel $b_1$ at 194 427
\pinlabel $b_2$ at 215 402
\pinlabel $b_3$ at 298 452
\pinlabel $b_4$ at 237 417
\pinlabel $b_5$ at 298 424
\pinlabel $\tau^{-}_3$ at 259 345
\pinlabel $b_1$ at 371 424
\pinlabel $b_2$ at 396 401
\pinlabel $b_3$ at 470 457
\pinlabel $b_4$ at 446 440
\pinlabel $b_5$ at 473 426
\pinlabel $\tau^{+}_3$ at 434 345 
\pinlabel $b_1$ at 12 270
\pinlabel $b_2$ at 38 254
\pinlabel $b_3$ at 119 295
\pinlabel $b_4$ at 64 257
\pinlabel $b_5$ at 123 263
\pinlabel $\tau^{--}_4$ at 88 183
\pinlabel $b_1$ at 186 271
\pinlabel $b_2$ at 266 289
\pinlabel $b_3$ at 279 204
\pinlabel $b_4$ at 236 255
\pinlabel $b_5$ at 297 264
\pinlabel $\tau^{-}_4$ at 259 183
\pinlabel $b_1$ at 362 271
\pinlabel $b_2$ at 411 292
\pinlabel $b_3$ at 466 204
\pinlabel $b_4$ at 449 276
\pinlabel $b_5$ at 495 275
\pinlabel $\tau^{+}_4$ at 434 183
\pinlabel $b_1$ at 12 109
\pinlabel $b_2$ at 41 102
\pinlabel $b_3$ at 115 136
\pinlabel $b_4$ at 66 80
\pinlabel $b_5$ at 122 100
\pinlabel $\tau^{--}_5$ at 88 21
\pinlabel $b_1$ at 186 109
\pinlabel $b_2$ at 266 134
\pinlabel $b_3$ at 284 41
\pinlabel $b_4$ at 235 89
\pinlabel $b_5$ at 297 101
\pinlabel $\tau^{-}_5$ at 259 21
\pinlabel $b_1$ at 362 109
\pinlabel $b_2$ at 399 130
\pinlabel $b_3$ at 466 41
\pinlabel $b_4$ at 453 98
\pinlabel $b_5$ at 485 103
\pinlabel $\tau^{+}_5$ at 434 21
\endlabellist 
\centering 
\includegraphics[width=\textwidth]{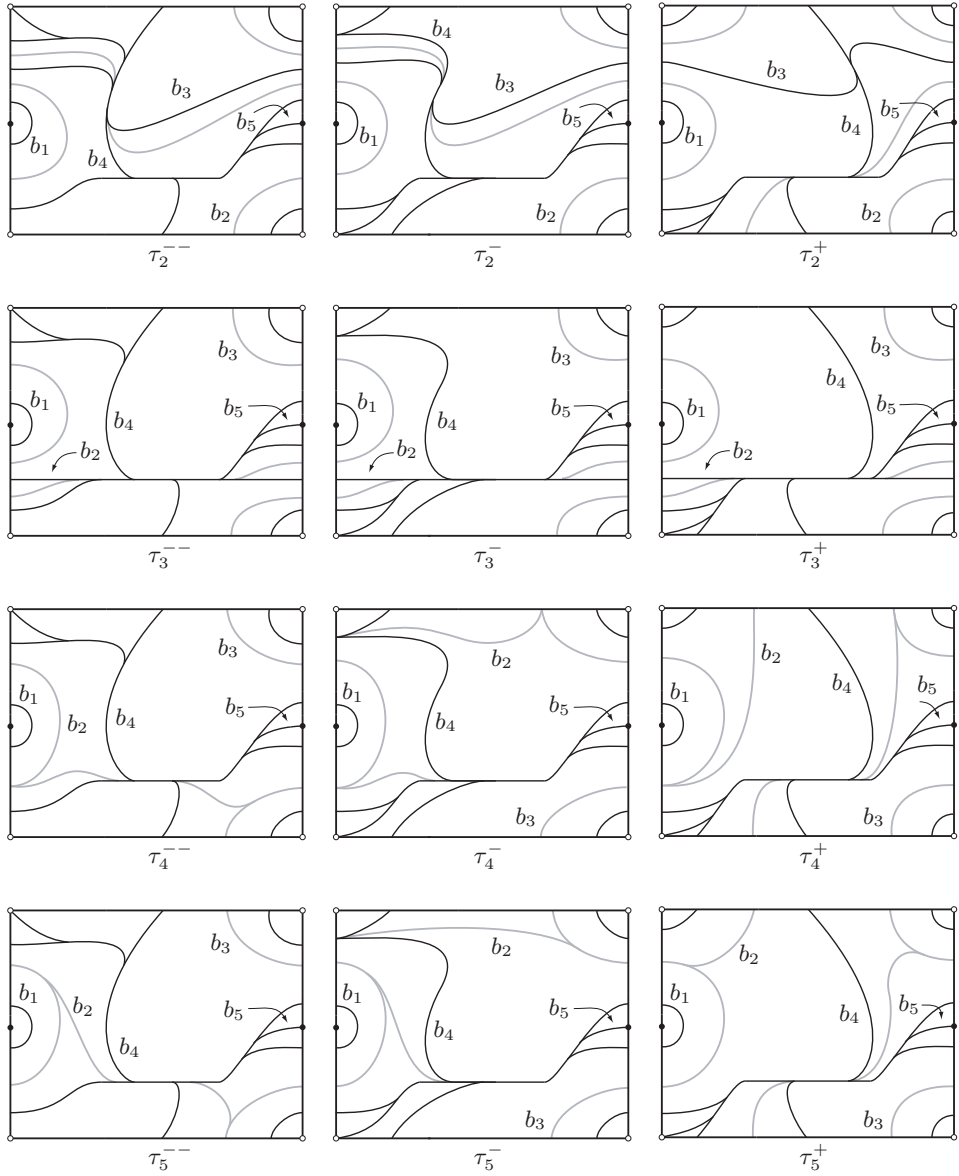} 
\caption{Maximal extensions $\tau=\tau_2,\tau_3,\tau_4,\tau_5$.\label{fig:extensions}}
\end{figure} 

\begin{proof}

Given that $\beta$ is an attaching sequence for $t_\beta$, $\beta$ determines (up to isotopy) an attaching circle (that we denote by $\beta$ also) that is disjoint from $t_\beta$.
Let $N=N(\supp t_\beta)$ be a tie neighborhood of the support $\supp t_\beta<\sigma$. 
Since $\beta\cap t_\beta=\emptyset$, we may assume that $\beta\cap \overline{N}$ consists of properly embedded arcs in $\overline{N}$, each of which passes between distinct vertical sides of $N$.
Furthermore, we take $\beta$ to be transverse to the ties of $N$ since $N$ is comprised of rectangles.

Given that $K$ is a nontrivial knot, the twist parameter $\rho$ is nonzero.
Suppose $\rho>0$.

If $s>0$ and $t>0$, then a component of $\Sigma-N$ is a polygon with at most three vertical sides. 
If an arc of $\beta\setminus N$ passes between two vertical sides of $N$ that are adjacent to a horizontal side in the frontier of $N$, then the arc is carried by $\supp t_\beta$. 
Otherwise there is an boundary arc of $\partial \hat{\Sigma}$ separating the vertical sides in the frontier of $N$ (this case corresponds to $r=0$), and the arc is carried by the extension $\supp t_\beta<\sigma$.
In either case $\beta\prec\sigma$.

If $s=0<t$, then there is a component of $\sigma-N$ with four vertical sides, and there may be an arc of $\beta\setminus N$ which is not parallel to a horizontal or terminal side. 
The hypothesis that $\beta$ is reduced implies, however, that each arc of $\beta\setminus (N\cup P)$ passes between distinct vertical sides of $N$ or components of $P\setminus N$, which we call \emph{passages}. 
Each component of $\Sigma\setminus (N\cup P)$ has at most three vertical sides or passages.
Thus, $\beta$ is carried by the extension $\supp t_\beta<\tau'$ obtained by adding branches between those vertical sides and passages not already connected by a horizontal side, provided we join all branches incident on a single passage. 
See Figure \ref{fig:presplitting}. 
Let $w=w_\beta$ be the counting measure of $\beta$ on $\tau'$.
From the definition of the intersection numbers and the switch conditions, we deduce $a=w_1$,  $b=w_1+w_2$, and $c=2w_1+w_3$.
If $2a\leq b$ ($2a>b$), then $w_1\leq w_2$ ($w_1>w_2$), and $\beta$ is carried by the right (left) splitting of $\tau'$ at branch $b_6$.
The right splitting is equivalent to $\sigma$ whereas the left splitting is the train track $\tau_2^+$ as required.

\begin{figure}[ht!] 
\labellist 
\footnotesize\hair 2pt 
\pinlabel $b_1$ at 39 86
\pinlabel $b_2$ at 69 79
\pinlabel $b_3$ at 58 114
\pinlabel $b_4$ at 105 71
\pinlabel $b_5$ at 29 36
\pinlabel $b_6$ at 142 44
\pinlabel $\tau'$ at 88 15 
\pinlabel $b_1$ at 224 96
\pinlabel $b_2$ at 217 69
\pinlabel $b_3$ at 230 129
\pinlabel $b_4$ at 266 95
\pinlabel $b_5$ at 322 93
\pinlabel $b_6$ at 186 125
\pinlabel $\tau''$ at 260 15 
\pinlabel $b_1$ at 372 84
\pinlabel $b_2$ at 416 85
\pinlabel $b_3$ at 383 132
\pinlabel $b_4$ at 455 119
\pinlabel $b_5$ at 493 92
\pinlabel $b_6$ at 494 54
\pinlabel $b_7$ at 362 107
\pinlabel $\tau'''$ at 432 15
\endlabellist 
\centering 
\includegraphics[width=\textwidth]{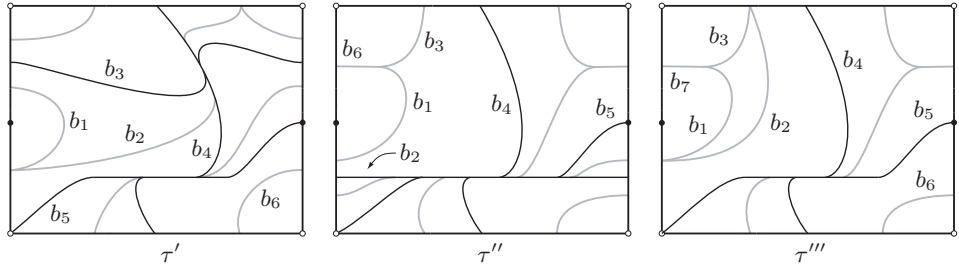} 
\caption{Extensions $\tau'$, $\tau''$, $\tau'''$ for $\rho>0.$\label{fig:presplitting}}
\end{figure}  

If $s>0=t$, a similar argument shows that $\beta$ is carried by the extension $\supp t_\beta<\tau''$ shown in Figure \ref{fig:presplitting}.
In this case $2a\leq b$.
If $2a\leq c$, then $\beta$ is carried by $\sigma$, which is the left splitting of $\tau''$ at $b_6$.
Otherwise $2a>c$, and $\beta$ is carried by $\tau_3^+$, which is the right splitting at $b_6$.

If $s=0=t$, then each component of $\Sigma\setminus (N\cup P)$ still has at most three vertical sides or passages, and $\beta$ is carried by the extension $\tau'''$ (Figure \ref{fig:presplitting}, right) which is defined in a manner analogous to $\tau'$ and $\tau''$. 
If $2a\leq b$ and $2a\leq c$ then $\beta$ is carried by the track resulting from splitting $\tau'''$ right at $b_6$ and left at $b_7$, which is equivalent to $\sigma$. 
If $2a>c\geq b$, then a left splitting at $b_6$ results in a train track equivalent to $\tau_4^+$, which carries $\beta$.
If $2a>b>c$, then $\beta$ is carried by the train track obtained by splitting $\tau'''$ right at $b_7$, which is equivalent to $\tau_5^+$.

In the remaining cases, $\rho<-r$ and $-r\leq\rho<0$, similar arguments demonstrate that $\beta$ is carried by $\tau$ as claimed.
\end{proof}

\begin{lem}\label{lem:heegform}
If $v=v_\beta$ is the counting measure of $\beta$ on the train track $\tau\neq\sigma$, where $\beta$ is an attaching circle of a $(1,1)$ knot $K$ with Schubert form $S(r,s,t,\rho)$, then a Heegaard diagram Schubert form $HS(r,'s',t',\rho')$ for $K$ is as follows.  
Let $\overline{x}$ denote $x\mod{v_2}$.

If $\tau=\tau_2$, then		
$$
(r',s',t',\rho')=
\begin{cases}

(v_1,v_3,v_2,-v_1-v_2-v_4) & \text{if $\rho<0$,}\\
(v_1,v_3,v_2,-2v_1+v_4) & \text{if $\rho>0$.}\\
\end{cases}
$$

If $\tau=\tau_3$, then
$$
(r',s',t',\rho')=
\begin{cases}
(v_1,v_3,v_2,-v_4-v_3-v_1) & \text{if $\rho<-r$,}\\
(v_1,v_3,v_2,2v_1+v_3-v_4) & \text{if $-r\leq\rho<0$,}\\
(v_1,v_3,v_2,2v_1+v_3+v_4) & \text{if $\rho\geq -r$.}
\end{cases}
$$

If $\tau=\tau_4$, then 		
$$
(r',s',t',\rho')=
\begin{cases}
(v_1,v_2-\overline{v_2-v_3},\overline{v_2-v_3},-\overline{v_2-v_3}-v_4) & \text{if $\rho<-r$},\\
(v_1,\overline{v_3},v_2-\overline{v_3},-v_2+\overline{v_3}-v_4) & \text{if $-r\leq\rho<0$,}\\
(v_1,\overline{v_3},v_2-\overline{v_3},v_2+\overline{v_3}+v_4) & \text{if $\rho>0$.}
\end{cases}
$$

If $\tau=\tau_5$, then 
$$
(r',s',t',\rho')=
\begin{cases}
(v_1,v_2-\overline{v_3},\overline{v_3},-\overline{v_3}-v_4) & \text{if $\rho<-r$,}\\
(v_1,v_2-\overline{v_3},\overline{v_3},-v_2+\overline{v_3}-v_4) & \text{if $-r\leq\rho<0$,}\\
(v_1,\overline{v_2-v_3},v_2-\overline{v_2-v_3},v_2+\overline{v_2-v_3}+v_4) & \text{if $\rho>0$.}
\end{cases}
$$
\end{lem}

\begin{proof}
Suppose $\tau=\tau_2$ and $\rho>0$. 
Apply an isotopy sliding the basepoint $y$ once around $\Sigma$ along a circular path parallel to $\mu$ and with the same orientation as $\mu$.
See Figure \ref{fig:large}.   
\begin{figure}[b!] 
\labellist 
\footnotesize\hair 2pt 
\pinlabel $b_1$ at 11 42
\pinlabel $b_2$ at 13 15
\pinlabel $b_3$ at 56 92
\pinlabel $b_4$ at 105 58
\pinlabel $b_5$ at 141 35
\pinlabel $b_1$ at 189 44
\pinlabel $b_2$ at 192 86
\pinlabel $b_3$ at 208 16
\pinlabel $b_4$ at 269 91
\pinlabel $b_5$ at 314 36
\pinlabel $b_1$ at 370 73
\pinlabel $b_2$ at 387 101
\pinlabel $b_3$ at 375 41
\pinlabel $b_4$ at 454 109
\pinlabel $b_5$ at 491 41
\endlabellist 
\centering 
\includegraphics[width=\textwidth]{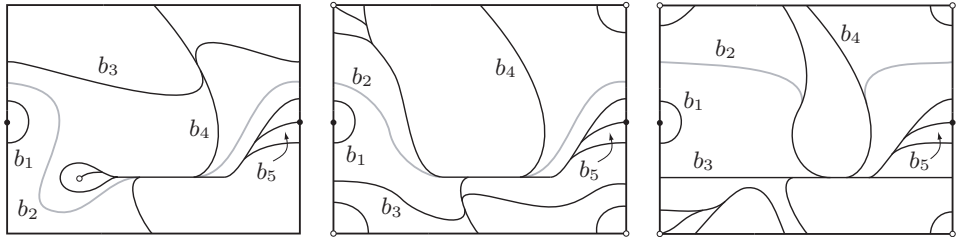} 
\caption{Isotopy of $\tau=\tau_2^+$.\label{fig:large}}
\end{figure} 
Split the large branch adjacent to $b_4$, and the result is a subtrack of $\sigma$ on which the counting measure of $\beta$ takes the values $w_1=v_1$, $w_2=v_3$, $w_3=v_2$, and $w_4=v_4-2v_1$, as required.
The same isotopy treats cases with $\rho<0$.

If $\tau=\tau_3$, slide the basepoint $y$ once around $\Sigma$ along a path parallel to $\mu$ and \emph{against} its orientation. 
The Heegaard diagram Schubert form follows by a splitting as in the above case. 

If $\tau=\tau_4$, then arcs of $\beta$ carried by $b_2$ wrap $v_3/v_2$-times around a meridian $\gamma$ parallel to $\mu$ with the same orientation.
The case $\rho>0$ is shown in Figure \ref{fig:dehntwist}, left, where, for sake of clarity, we have isotoped the basepoint $y$ toward the center of the fundamental polygon.
In this case, if we apply a $-\lfloor v_3/v_2 \rfloor$ Dehn twist to $\beta$ along $\gamma$, where $\lfloor x\rfloor$ denotes the greatest integer less than or equal to $x$, then the result $\beta'$ is carried by the train track $\tau'$ shown in Figure \ref{fig:dehntwist}, center.
Since $\gamma$ is disjoint from $\alpha$, $x$, and $y$, and the twist is integral, the doubly-pointed Heegaard diagrams $(\Sigma,\alpha,\beta,x,y)$ and $(\Sigma,\alpha,\beta',x,y)$ are equivalent; i.e., they represent the same knot in $S^3$.
If $v'$ is the counting measure of $\beta'$ induced by $v$, then $(v'_1,\dots,v'_4)=(v_1,v_2,\overline{v_3},v_4)$, where $\overline{x}=x\mod v_2$.
By isotoping $\tau'$ into a neighborhood of $\sigma$, as shown in Figure \ref{fig:dehntwist}, we obtain the Heegaard diagram Schubert form $HS(v_1',v_3',v_2'-v_3',v_2'+v_3'+v_4')=HS(v_1,\overline{v_3},v_2-\overline{v_3},v_2+\overline{v_3}+v_4)$, as required.
The case $-r\leq\rho<0$ follows similarly. 
If $\rho<-r$, an additional twist is required for Schubert form; after a $-\lceil v_3/v_2\rceil$ Dehn twist the given Heegaard diagram Schubert form is obtained, where $\lceil x\rceil$ denotes the smallest integer greater than $x$.

\begin{figure}[ht!] 
\labellist 
\footnotesize\hair 2pt 
\pinlabel $b_1$ at 66 55
\pinlabel $b_2$ at 69 109
\pinlabel $b_3$ at 47 19
\pinlabel $b_4$ at 117 90
\pinlabel $b_5$ at 141 42
\pinlabel $\gamma$ at 12 29
\pinlabel $b'_1$ at 242 56
\pinlabel $b'_2$ at 245 111
\pinlabel $b'_3$ at 203 24
\pinlabel $b'_4$ at 292 89
\pinlabel $b'_5$ at 315 42
\pinlabel $b'_1$ at 371 73
\pinlabel $b'_2$ at 424 111
\pinlabel $b'_3$ at 372 47
\pinlabel $b'_4$ at 437 54
\pinlabel $b'_5$ at 488 45
\endlabellist 
\centering 
\includegraphics[width=\textwidth]{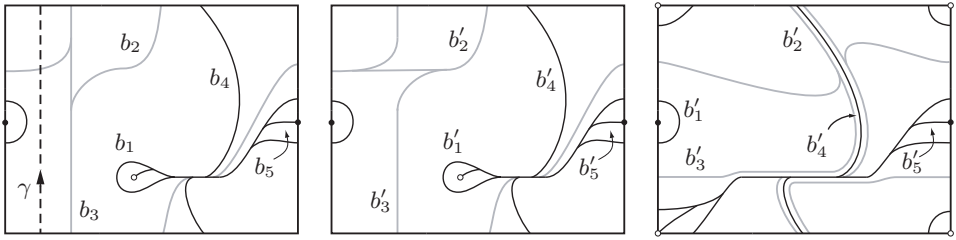} 
\caption{Dehn twist of $\tau=\tau_4^+$ along $\gamma$.}\label{fig:dehntwist}
\end{figure} 

If $\tau=\tau_5$ and $\rho<-r$ ($\rho\geq -r$), then a \emph{positive} $\lfloor v_3/v_2\rfloor$ ($\lceil v_3/v_2 \rceil$) Dehn twist along $\gamma$ takes $\beta$ to a neighborhood of $\sigma$. 
The required Heegaard diagram Schubert parameters follow by an argument similar to the one applied in the $\tau_4$ case.
\end{proof}

\begin{exmp}[Trefoil Heegaard diagram Schubert form]

From Example \ref{ex:trefoilsequence}, an attaching sequence for the trefoil with Schubert form $S(0,0,2,2)$ is 
$$
\beta=(6.5, 2.5, 5.5, 1.5, 4.5, 13.5, 9.5, 13.5, 9.5).
$$
Given the trefoil basepoints $(y_1,y_2,x_1,y_3,y_4,x_2)=(0,3,4,7,10,13)$ from Example \ref{ex:trefoilbridge}, the intersection numbers of $\beta$ are $(a,b,c,d)=(1,2,3,4)$. 
Since $2a\leq b$, $2a\leq c$, and $\rho\geq -r$,  a Heegaard diagram Schubert form for the trefoil is
$$
HS(r',s',t',\rho')=(1,0,1,2).
$$
To verify that this corresponds to a Heegaard diagram compatible with the trefoil, notice that an isotopy sliding $y$ in the direction of the path $\mu\lambda^{-1}$ takes $\beta$ into the shape of Figure \ref{fig:trefoil}.
\end{exmp}


\begin{thebibliography}{GMM05}

\bibitem[CM04]{Cattabriga-Mulazzani}
Alessia Cattabriga and Michele Mulazzani.
\newblock All strongly-cyclic branched coverings of {$(1,1)$}-knots are
  {D}unwoody manifolds.
\newblock {\em J. London Math. Soc. (2)}, 70(2):512--528, 2004.

\bibitem[CK03]{Choi-Ko}
Doo~Ho Choi and Ki~Hyoung Ko.
\newblock Parametrizations of 1-bridge torus knots.
\newblock {\em J. Knot Theory Ramifications}, 12(4):463--491, 2003.

\bibitem[Dol92]{Doll}
H.~Doll.
\newblock A generalized bridge number for links in $3$-manifolds.
\newblock {\em Math. Ann.}, 294(0025-5831):701--717, 1992.

\bibitem[Doy05]{Doyle}
Gabriel Doyle.
\newblock solve11knot --- a perl program for calculating the knot floer
  homology of a doubly-pointed genus 1 heegaard knot diagram, 2005.

\bibitem[Fuj96]{Fujii}
Hirozumi Fujii.
\newblock Geometric indices and the {A}lexander polynomial of a knot.
\newblock {\em Proc. Amer. Math. Soc.}, 124(9):2923--2933, 1996.

\bibitem[GMM05]{Goda-Matsuda-Morifuji}
Hiroshi Goda, Hiroshi Matsuda, and Takayuki Morifuji.
\newblock Knot {F}loer homology of {$(1,1)$}-knots.
\newblock {\em Geom. Dedicata}, 112:197--214, 2005.

\bibitem[Hed]{Hedden:Private}
Matthew Hedden.
\newblock Private communication.

\bibitem[MMS]{Masur-Mosher-Schleimer}
Howard Masur, Lee Mosher, and Saul Schleimer.
\newblock On train track splitting sequences.
\newblock arXiv:1004.4564v1 [math.GT].

\bibitem[Ord06]{Ording}
Philip Ording.
\newblock {\em On knot Floer homology of satellite (1,1) knots}.
\newblock PhD thesis, Columbia University, 2006.

\bibitem[OS06]{Ozsvath-Szabo:Intro}
Peter Ozsv{\'a}th and Zolt{\'a}n Szab{\'o}.
\newblock An introduction to {H}eegaard {F}loer homology.
\newblock In {\em Floer homology, gauge theory, and low-dimensional topology},
  volume~5 of {\em Clay Math. Proc.}, pages 3--27. Amer. Math. Soc.,
  Providence, RI, 2006.

\bibitem[PH92]{Penner-Harer}
R.~C. Penner and J.~L. Harer.
\newblock {\em Combinatorics of train tracks}, volume 125 of {\em Annals of
  Mathematics Studies}.
\newblock Princeton University Press, Princeton, NJ, 1992.

\bibitem[Ras]{Rasmussen:KnotHomologies}
Jacob~A. Rasmussen.
\newblock Knot polynomials and knot homologies.
\newblock 	arXiv:math/0504045v1 [math.GT].

\bibitem[Sin33]{Singer}
James Singer.
\newblock Three-dimensional manifolds and their {H}eegaard diagrams.
\newblock {\em Trans. Amer. Math. Soc.}, 35(1):88--111, 1933.

\bibitem[Thu80]{Thurston}
William Thurston.
\newblock The geometry and topology of three-manifolds.
\newblock 1980.

\end{thebibliography}
\end{document}